\documentclass[12pt]{article}
\usepackage{amsfonts, amsmath, amsthm, amssymb, relsize, graphics, tabu}
\usepackage[table]{xcolor}
\usepackage [english]{babel}
\usepackage{cjhebrew}
\usepackage{mathrsfs}
\usepackage{enumerate}
\usepackage [autostyle, english = american]{csquotes}
\usepackage{harpoon}

\newtheorem{thm}{Theorem}
\newtheorem{lem}{Lemma}

\title{Existence of Some Signed Magic Arrays\thanks{Research
supported by NSF REU Grant DMS1262838, University of West Georgia}
}
\author{
Abdollah Khodkar\\
Department of Mathematics\\
University of West Georgia\\
Carrollton, GA 30118\\
{\tt akhodkar@westga.edu}\vspace{3mm}\\
Christian Schulz\\
Department of Mathematics\\
Rose Hulman Institute of Technology\\
5500 Wabash Ave, Terre Haute, IN 47803\\
{\tt schulzcc@rose-hulman.edu}\vspace{3mm}\\
Nathan Wagner\\
Department of Mathematics\\
Bucknell University\\
701 Moore Avenue,Lewisburg, PA 17837\\
{\tt naw006@bucknell.edu}
}

\begin{document}

\maketitle

\begin{abstract}
	We consider the notion of a signed magic array, which is an $m \times n$ rectangular array with the same number of filled cells $s$ in each row and the same number of filled cells $t$ in each column, filled with a certain set of numbers that is symmetric about the number zero, such that every row and column has a zero sum. We attempt to make progress toward a characterization of for which $(m, n, s, t)$ there exists such an array. This characterization is complete in the case where $n = s$ and in the case where $n = m$; we also characterize three-fourths of the cases where $n = 2m$.

\end{abstract}

\vspace{10pt}

\section{Introduction}\label{introduction}

A {\em magic rectangle} is defined as an $m \times n$ array whose entries are precisely the integers from $0$ to $mn-1$ wherein the sum of each row is $c$ and the sum of each column is $r$.
A {\em magic square} is a magic rectangle with $m=n$ and $c=r$.
In \cite{sun} it is proved that:

\begin{thm}\label{Th:Sun}
There is an $m \times n$ magic rectangle if and only if $m \equiv n \pmod 2$, $m + n > 5$, and $m, n > 1$.
\end{thm}

An {\em integer Heffter array} $H(m, n; s, t)$ is an $m\times n$ array with entries from
$X=\{\pm1,\pm2,$ $\ldots,\pm ms\}$
such that each row contains $s$ filled cells and each column contains $t$ filled cells,
the elements in every row and column sum to 0 in ${\mathbb Z}$, and
for every $x\in A$, either $x$ or $-x$ appears in the array.
The notion of an integer Heffter array $H(m, n; s, t)$ was first defined by Archdeacon in \cite{arc1}.
Integer Heffter arrays with $m=n$ represent a type of magic square where each number from the set
$\{1,2,\ldots,n^2\}$ is used once up to sign. A Heffter array is {\em tight} if it has no empty cell; that is,
$n=s$ (and necessarily $m = t$).

\begin{thm}\label{tightHeffter}\cite{arc2}
Let $m, n$ be integers at least 3.
There is a tight integer Heffter array if and only if $mn\equiv 0, 3 \pmod 4$.
\end{thm}

A {\em square} integer Heffter array $H(n; k)$ is an integer Heffter array with $m=n$ and $s=t=k$.
In \cite{ADDY,DW} it is proved that:

\begin{thm}\label{Heffterwithemptycells}
There is an integer $H(n; k)$ if and only if $3\leq k\leq n$ and $nk\equiv 0,3 \pmod 4$.
\end{thm}

A {\em signed magic array} $SMA(m,n;s,t)$ is an $m \times n$ array with entries from $X$, where
$X=\{0,\pm1,\pm2,\ldots,\pm (ms-1)/2\}$ if $ms$ is odd and $X = \{\pm1,\pm2,\ldots,\pm ms/2\}$ if $ms$ is even,
such that precisely $s$ cells in every row and $t$ cells in every column are filled,
every integer from set $X$ appears exactly once in the array and
the sum of each row and of each column is zero.
In the case where $m = n$, we call the array a {\em signed magic square}.
Signed magic squares also represent a type of magic square where each number from the set $X$
is used once.

We use the notation $SMS(n;t)$ for a signed magic square with $t$ filled cells in each row and
$t$ filled cells in each column. An $SMS(n;t)$ is called $k$-{\em diagonal} if its entries all belong to $k$ consecutive diagonals (this includes broken diagonals as well). In the case where $k = t$, we abbreviate this to simply {\em diagonal}.
An $SMA(m,n;s,t)$ is called {\em tight}, and denoted $SMA(m,n)$, if it contains no empty cells; that is $m=t$ and $n=s$. Figure \ref{3in53by4} displays two examples of signed magic arrays.

\begin{figure}[ht]
$$\begin{array}{ccccc}
	\begin{array}{|c|c|c|c|c|}
	\hline
	2 & 3 & & & -5 \\ \hline
	-7 & 1 & 6 & & \\ \hline
	& -4 & 0 & 4 & \\ \hline
	& & -6 & -1 & 7 \\ \hline
	5 & & & -3 & -2 \\ \hline
	\end{array}&&&&	
	\begin{array}{|c|c|c|c|}
	\hline
	1 & -1 & 2 & -2 \\ \hline
	5 & 4 & -5 & -4 \\ \hline
	-6 & -3 & 3 & 6 \\ \hline
	\end{array}\\
\end{array}$$
\caption{A diagonal $SMS(5; 3)$ and an $SMA(3, 4)$.}
	\label{3in53by4}
\end{figure}



In this paper we investigate the existence of $SMA(m,n)$, $SMS(m;t)$ and $SMA(m,$ $2m;2t,t)$.
In Section \ref{TSMA} we prove an $SMA(m,n)$ exists precisely when $m = n = 1$, or when $m = 2$ and $n \equiv 0, 3 \pmod4$, or when $n = 2$ and $m \equiv 0, 3 \pmod4$, or when $m, n > 2$.
In Section \ref{SMS} we
show that there exists an $SMS(n;t)$ for $n \geq t \geq 1$ precisely when $n, t = 1$ or $n, t > 2$.
Finally, in Section \ref{SMR} we prove that there exists an $SMA(m,2m;2t,t)$ if $m \geq t \geq 3$ and $mt \equiv 0 \text{ or } 3 \pmod 4$ or $m, t \equiv 2 \pmod 4$.

In the following sections, the notation $[a,b]$ refers to the set of integers $z$ such that $a\leq z\leq b$.
Two partitions ${\mathcal P}_1$ and ${\mathcal P}_2$ of a set $A$ are {\em orthogonal} if the intersection
of each member of ${\mathcal P}_1$ and of ${\mathcal P}_2$ has precisely one element.

A rectangular array is {\em shiftable} if it contains the same number of positive as negative entries in every column and in every row (see \cite{arc1}).
These arrays are called \textit{shiftable} because they may be shifted to use different absolute values. By increasing the absolute value of each entry by $k$, we add $k$ to each positive entry and $-k$ to each negative entry. If the number of entries in a row is $2\ell$, this means that we add $\ell k + \ell(-k) = 0$ to each row, and the same argument applies to the columns. Thus, when shifted, the array retains the same row and column sums.

\section{Tight signed magic arrays}\label{TSMA}

We first examine the case of a tight array, with all of its cells filled. We will completely characterize the values of $m$ and $n$ for which tight $m \times n$ signed magic arrays exist through the use of several lemmata. The proof of the following lemma is trivial.

\begin{lem}\label{1xn}
	A tight $SMA(1,n)$ exists if and only if $n = 1$.
\end{lem}

\begin{lem}\label{2xn}
	An $SMA(2,n)$ exists if and only if $n \equiv 0, 3 \pmod4$.
\end{lem}

\begin{proof}
	In an $SMA(2,n)$, let $x$ be a value in a column; then $-x$ must be the other value if their sum is zero. Thus, each row in the $2 \times n$ array contains every absolute value from $1$ to $n$ exactly once. If $n \equiv 1 \pmod4$, then $n(n+1) \equiv 2 \pmod4$, so $\sum_{i=1}^{n} i = \frac{1}{2}n(n+1)$ is odd. Note that for all $x$, $x - (-x) = 2x$ is even, so by replacing any number by its negative in a sum, one cannot change the parity of the sum. Because each row in the $2 \times n$ array contains every absolute value from $1$ to $n$ exactly once, their sum is equal to $\sum_{i=1}^{n} i$ after some of the positives have been replaced by negatives; but this sum will always be odd, so the sum of a row cannot be the even number 0. Thus, no tight $2 \times n$ signed magic array exists. If $n \equiv 2 \pmod4$, then $n(n+1) \equiv 2 \pmod4$, and the same argument holds as in the previous case.
	
Now let $n \equiv 0, 3 \pmod4$. By induction we prove that an $SMA(2,n)$ exists.  	
A $2\times 3$ array with the first row $1,2,-3$ and the second row $-1,-2,3$ is obviously an $SMA(2,3)$.
When $n=4$ we use the array given in Figure \ref{2x4}.
	
	\begin{figure}[ht]
		$$\begin{array}{| c | c | c | c |}
		\hline
		1 & -2 & -3 & 4 \\ \hline
		-1 & 2 & 3 & -4 \\ \hline
		\end{array}$$
		\caption{A shiftable $SMA(2,4)$}
		\label{2x4}
	\end{figure}
	
	Now let $n \equiv 0, 3 \pmod4$ with $n > 4$, and assume the existence of an $SMA(2,n - 4)$. Onto the right side of this array we add four columns, leaving a $2 \times 4$ space of cells to be filled. Note that the $2 \times 4$ array in Figure \ref{2x4} is shiftable, so we may merely shift its absolute values from $1$ through $4$ to $n - 3$ through $n$ and use it to fill the empty space. As the shiftable array has zero row and column sums, the sum of each row is still zero, and the sum of each of the four new columns is also zero. We thus construct an $SMA(2,n)$.
Now the result follows by strong induction.
\end{proof}

The remaining cases split based on the parity of $m$ and $n$.

\subsection{Tight signed magic arrays with $m, n$ both even}

\begin{lem}\label{evenxeven}
	A shiftable $SMA(m,n)$ exists if $m, n$ are both even and greater than $2$.
\end{lem}

\begin{proof}
	Proceed by strong induction first on $n$ and then on $m$. As the base case, we provide arrays for $(m, n) = (4, 4)$, $(4, 6)$, and $(6, 6)$ in Figures \ref{4x4,4x6} and \ref{6x6}. Note that as the transpose of a signed magic array is a signed magic array, we need not provide a separate $6 \times 4$ array and may transpose the $4 \times 6$ array.
	
	\begin{figure}[ht]
   $$\begin{array}{cccc}
		\begin{array}{| c | c | c | c |}
		\hline
		1 & -2 & -3 & 4 \\ \hline
		-1 & 2 & 3 & -4 \\ \hline
		5 & -6 & -7 & 8 \\ \hline
		-5 & 6 & 7 & -8 \\ \hline
		\end{array}&&&
		\begin{array}{| c | c | c | c | c | c |}
		\hline
		1 & -2 & -3 & 4 & 9 & -9 \\ \hline
		-1 & 2 & 3 & -4 & -10 & 10 \\ \hline
		5 & -6 & -7 & 8 & -11 & 11 \\ \hline
		-5 & 6 & 7 & -8 & 12 & -12 \\ \hline
		\end{array}\\
\end{array}$$
\caption{A shiftable $SMA(4,4)$ and a shiftable $SMA(4,6)$}
\label{4x4,4x6}
	\end{figure}
	
	\begin{figure}[ht]
		$$\begin{array}{| c | c | c | c | c | c |}
		\hline
		6 & -4 & -12 & -3 & 2 & 11 \\ \hline
		-13 & 15 & 16 & 7 & -8 & -17 \\ \hline
		10 & -18 & -5 & -14 & 18 & 9 \\ \hline
		-9 & 1 & 14 & 5 & -1 & -10 \\ \hline
		17 & 8 & -16 & -7 & -15 & 13 \\ \hline
		-11 & -2 & 3 & 12 & 4 & -6 \\ \hline
		\end{array}$$
		\caption{A shiftable $SMA(6,6)$}
		\label{6x6}
	\end{figure}
	
	Now, let $m \in \{4, 6\}$ and $n$ be even, and assume that there exists a shiftable $SMA(m,$ $n - 4)$. We may extend this array by adding four columns to create an $m \times n$ array. Note that the empty $m \times 4$ space, because $m$ is even, partitions into $2 \times 4$ rectangles, each of which may be filled by a shifted copy of Figure \ref{2x4} using a method analogous to that of Lemma \ref{2xn} with the appropriate absolute values. As the shifted copies each have a row and column sum of zero, they do not change the row sums from the $m \times (n - 4)$ array, and the sums of the new columns will be zero as well. Therefore, a shiftable $SMA(m, n)$ exists. Hence, by strong induction on $n$,
a shiftable $SMA(m, n)$ exists for $m \in \{4, 6\}$ and $n > 2$ even.
	
	Now, let $m$ and $n$ both be even, and assume that there exists a shiftable $SMA(m - 4,n)$. We may extend this array by adding four rows to create an $m \times n$ array. Note that the empty $4 \times n$ space, because $n$ is even, partitions into $4 \times 2$ rectangles, each of which may be filled by a shifted copy of the transpose of Figure \ref{2x4} using a method analogous to that of the previous paragraph. As the shifted copies each have a row and column sum of zero, they do not change the column sums from the $m \times (n - 4)$ array, and the sums of the new rows will be zero as well. Therefore, a tight shiftable $SMA(m, n)$ exists. By strong induction on $m$, a shiftable $SMA(m,n)$ always exists for $m > 2, n > 2$ even.
\end{proof}

\subsection{Tight signed magic arrays with $m, n$ both odd}

Recall that a magic rectangle is defined as an $m \times n$ array whose entries are precisely the integers from $0$ to $mn-1$ wherein the sum of each row is $c$ and the sum of each column is $r$.

\begin{lem}\label{oddxodd}
	A $SMA(m, n)$ exists if $m, n$ are both odd and greater than $1$.
\end{lem}

\begin{proof}
	
	Let $m, n$ be odd and greater than $1$. Then $m + n \geq 6$. Hence by Theorem \ref{Th:Sun} there exists an $m \times n$ magic rectangle, say $A$, and let $a_{r, c}$ be the entry in row $r$ and column $c$ of $A$.
	
	We will then construct an array $B$ wherein $b_{r, c} = a_{r, c} - w$, where $w = \frac{mn-1}{2}$. As the entries in $A$ are precisely the integers $0$ through $mn-1$, it follows that the entries in $B$ are precisely the integers $-\frac{mn-1}{2}$ through $\frac{mn-1}{2}$, the required set of integers for a tight signed magic array. It remains to be shown that $B$ has rows and columns summing to zero.
	
	If the sum of column $c$ in $A$ is $s$, then the sum of column $c$ in $B$ is $s - mw$, as we subtract $w$ from each of the $m$ entries in the column. In particular, note that as $s$ is constant, this entire expression is independent of $c$; so the sum of each column in $B$ is the same. The sum of all of the entries in $B$ is $\sum_{i=-w}^{w} i = \sum_{i=-w}^{-1} i + 0 + \sum_{i=1}^{w} i = -\frac{1}{2}w(w+1) + 0 + \frac{1}{2}w(w+1) = 0$. If the sum of each column is the same, and the sum of all of the columns together is zero, then the sum of each column must be zero.
	
	Likewise, if the sum of row $r$ in $A$ is $s$, then the sum of row $r$ in $B$ is $s - nw$, as we subtract $w$ from each of the $n$ entries in the row. In particular, note that as $s$ is constant, this entire expression is independent of $r$; so the sum of each row in $B$ is the same. The sum of all of the entries in $B$ is 0, so if the sum of each row is the same, then that sum must be zero. Hence, $B$ is an $SMA(m,n)$, where $m, n$ are odd and greater than 1.
\end{proof}
	

\subsection{Tight signed magic arrays with $m$ odd, $n$ even}

For this case, we will need to make use of an induction argument with two base cases. The base cases are given in the following lemmata.

\begin{lem} \label{3xeven}
	An $SMA(3,n)$  exists if $n$ is even.
\end{lem}

\begin{proof}
An $SMA(3,2)$ and an $SMA(3,4)$ are given in Figure \ref{3x2.3x4}.

\begin{figure}[ht]
$$\begin{array}{cccc}
    \begin{array}{|c|c|}
	\hline
	1 & -1 \\\hline
	2 & -2 \\\hline
	-3 & 3 \\\hline
	\end{array}&&&
	\begin{array}{|c|c|c|c|}\hline
	1 & -1 & 2 & -2 \\\hline
	5 & 4 & -5 &-4 \\\hline
	-6 & -3 & 3 & 6 \\\hline
	\end{array}\\
\end{array}$$
\caption{An $SMA(3,2)$ and an $SMA(3,4)$}
		\label{3x2.3x4}
	\end{figure}

    Now let $n=2k\geq 6$ and $p_{j}=\lceil\frac{j}{2}\rceil$ for $1\leq j\leq 2k$ . Define a $3\times n$ array $A=[a_{i,j}]$ as follows. For $1\leq j\leq 2k$,

$$a_{1,j}= \begin{cases}
	- \left(\frac{3p_{j}-2}{2}\right) & j \equiv 0 \pmod4 \\
	\frac{3p_{j}-1}{2}& j \equiv 1 \pmod4 \\
	-\left(\frac{3p_{j}-1}{2}\right) & j \equiv 2 \pmod4 \\
	\frac{3p_{j}-2}{2} & j \equiv 3 \pmod4. \\
	
\end{cases}$$
For the third row we define $a_{3,1}=-3k$, $a_{3,2k}=3k$ and when $2\leq j\leq 2k-1$

$$a_{3,j}= \begin{cases}
	- 3(k-p_{j}) & j \equiv 0\pmod4 \\
	3(k-p_{j}+1) & j \equiv 1 \pmod4 \\
	-3(k-p_{j}) & j \equiv 2 \pmod4 \\
	3(k-p_{j}+1) & j \equiv 3 \pmod4. \\
\end{cases}$$
Finally, $a_{2,j}=-(a_{1,j}+a_{3,j})$ for $1\leq j\leq2k$ (see Figure \ref{3x10}).
It is straightforward to see that array $A$ is an $SMA(3,n)$.
\end{proof}

\begin{figure}[ht]
$$\begin{array}{|c|c|c|c|c|c|c|c|c|c|}\hline
1&-1&2&-2&4&-4&5&-5&7&-7 \\\hline
14&13&-14&11&-13&10&-11&8&-10&-8	\\\hline
-15&-12&12&-9&9&-6&6&-3&3&15\\\hline
\end{array}$$
\caption{An $SMA(3,10)$ using the method given in Lemma \ref{3xeven}.}
		\label{3x10}
\end{figure}

\begin{lem}\label{5xeven}
	An $SMA(5,n)$ exists if $n$ is even and greater than $2$.
\end{lem}

\begin{proof}
	If $n$ is a multiple of 4, we first apply Lemma \ref{3xeven} to construct an $SMA(3,n)$. We then adjoin two more rows to the bottom of this array, creating a $2 \times n$ space. As $n$ is a multiple of $4$, we can fill this space with shifted copies of Figure \ref{2x4} such that the sum of each column remains zero and the sums of rows 4 and 5 are also zero (see Figure \ref{5x4}).
	
	\begin{figure}[ht]
		$$\begin{array}{| c | c | c | c |}
		\hline
		1 & -1 & 2 & -2 \\ \hline
		5 & 4 & -5 & -4 \\ \hline
		-6 & -3 & 3 & 6 \\ \hline
		\cellcolor{gray!25}7 & \cellcolor{gray!25}-8 & \cellcolor{gray!25}-9 & \cellcolor{gray!25}10 \\ \hline
		\cellcolor{gray!25}-7 & \cellcolor{gray!25}8 & \cellcolor{gray!25}9 & \cellcolor{gray!25}-10 \\ \hline
		\end{array}$$
		\caption{An $SMA(5,4)$ using the method given in Lemma \ref{5xeven} when $n\equiv 0 \pmod 4$.}
		\label{5x4}
	\end{figure}
	
	If $n$ is not a multiple of 4, we again use the algorithm of Lemma \ref{3xeven} to first construct an
 $SMA(3,n)$, say $A$. This array will use numbers with absolute value $1$ through $\frac{3n}{2}$.
Note that the first two entries in the top row of this array will be $1$ and $-1$, and provided $n > 2$, the first two entries in the bottom row will be $-\frac{3n}{2}$ and $-\frac{3n}{2} + 3$. Using the fact that each column sums to zero gives us that the first two entries in the middle row are $x_1 = \frac{3n}{2} - 1$ and $x_2 = \frac{3n}{2} - 2$. Of importance is the fact that $x_1 - x_2 = 1$.
	
	Now we will construct an $SMA(5,n)$ as follows. The rows of array $A$ are placed in the first three rows of this array, with the exception that $x_1$ and $x_2$ are swapped; that is, the first two entries in the second row are $x_2, x_1$ instead of $x_1, x_2$. The bottom two rows, ignoring the left two columns, form a $2 \times (n-2)$ array; as $n-2$ is a multiple of 4, this array can be tiled with shifted copies of Figure \ref{2x4}, using absolute values from $\frac{3n}{2} + 3$ to $\frac{5n}{2}$ (note that this includes $(\frac{5n}{2}) - (\frac{3n}{2} + 3) + 1 = n - 2$ consecutive absolute values). This leaves four cells in the lower left, which may then be filled as follows:
	
	\begin{center}
		{\tabulinesep=2mm
			\begin{tabu}{| c | c |}
				\hline
				$-\frac{3n}{2}-1$ & $\frac{3n}{2}+1$ \\ \hline
				$\frac{3n}{2}+2$ & $-\frac{3n}{2}-2$ \\ \hline
			\end{tabu}
		}
	\end{center}

	Figure \ref{5x6} gives an example of this construction.

	Now we will prove that the resulting array is an $SMA(5,n)$. First, we note that we have in fact used every absolute value from 1 to $\frac{5n}{2}$ exactly once as a positive and once as a negative value. Now consider the sum of a given row. For row 1 and row 3, the sum is zero immediately from Lemma \ref{3xeven}. For row 2, the sum is zero because permuting the values in a row does not change their sum. For row 4 and row 5, one may observe that the left two columns cancel each other's values, and the rest of the row is filled with shifted copies of Figure \ref{2x4} guaranteed to sum to zero.
	
	Lastly, we consider the sums of the columns. For all but the first two columns, the sum of the first three rows will be zero from Lemma \ref{3xeven} and the fact that the last two values cancel each other in the corresponding copy of Figure \ref{2x4}. In the first column, the sum can be computed by calculating the sum of the differences from the known $3 \times n$ solution, as $x_2 - x_1 + (-\frac{3n}{2}-1) + (\frac{3n}{2}+2) = -1 - 1 + 2 = 0$. In the second column, we may similarly compute the sum to be $x_1 - x_2 + (\frac{3n}{2}+1) + (-\frac{3n}{2}-2) = 1 + 1 - 2 = 0$.
	This completes the proof.
\end{proof}

	\begin{figure}[ht]
		$$\begin{array}{| c | c | c | c | c | c |}
		\hline
		1 & -1 & 2 & -2 & 4 & -4 \\ \hline
		\mathbf{7} & \mathbf{8} & -8 & 5 & -7 & -5 \\ \hline
		-9 & -6 & 6 & -3 & 3 & 9 \\ \hline
		-10 & 10 & \cellcolor{gray!25}12 & \cellcolor{gray!25}-13 & \cellcolor{gray!25}-14 & \cellcolor{gray!25}15 \\ \hline
		11 & -11 & \cellcolor{gray!25}-12 & \cellcolor{gray!25}13 & \cellcolor{gray!25}14 & \cellcolor{gray!25}-15 \\ \hline
		\end{array}$$
		\caption{A tight $5 \times 6$ signed magic array using the method given in Lemma \ref{5xeven}.}
		\label{5x6}
	\end{figure}

\begin{lem}\label{oddxeven}
There exists an $SMA(m,n)$ for all odd $m > 1$ and even $n > 2$.
\end{lem}
\begin{proof}
	For $m = 3,5$, we apply Lemma \ref{3xeven} and Lemma \ref{5xeven}, respectively.
	
	If $m > 5$, assume inductively that there exists an $SMA(m - 4, n)$. We then augment this array by adjoining four rows at the bottom, leaving an empty $4 \times n$ space. Place a shifted $SMA(4,n)$ from Lemma \ref{evenxeven} in the $4 \times n$ empty space at the bottom of the array. Then each column of the resulting $m \times n$ array sums to zero, as the first $m-4$ entries in the column sum to zero by assumption, and the last four entries sum to zero due to the shiftable array. Each row also clearly sums to zero. Therefore, there exists an $SMA(m, n)$. Hence, the statement is true by strong induction.
\end{proof}

We are now ready to state the main theorem of this section.
\begin{thm}
	\label{tight}
	An $SMA(m,n)$ exists precisely when $m = n = 1$, or when $m = 2$ and $n \equiv 0, 3 \pmod4$, or when $n = 2$ and $m \equiv 0, 3 \pmod4$, or when $m, n > 2$.
\end{thm}
\begin{proof}
	Note that the transpose of a signed magic array is a signed magic array, so the existence of an
 $SMA(m \times n)$ ensures that of an $SMA(n \times m)$.
 The result then follows immediately from the conjunction of several lemmata, as specified in
 Figure \ref{Tight array cases}.

	\begin{figure}
	\begin{center}
		{\tabulinesep=2mm
			\begin{tabu}{ c | c | c | c | c |}
				& $n = 1$ & $n = 2$ & $n > 2$ odd & $n > 2$ even \\ \hline
				$m = 1$ & Lemma \ref{1xn} & Lemma \ref{1xn} & Lemma \ref{1xn} & Lemma \ref{1xn} \\ \hline
				$m = 2$ & Lemma \ref{1xn} & Lemma \ref{2xn} & Lemma \ref{2xn} & Lemma \ref{2xn} \\ \hline
				$m > 2$ odd & Lemma \ref{1xn} & Lemma \ref{2xn} & Lemma \ref{oddxodd} & Lemma \ref{oddxeven} \\ \hline
				$m > 2$ even & Lemma \ref{1xn} & Lemma \ref{2xn} & Lemma \ref{oddxeven} & Lemma \ref{evenxeven} \\ \hline
				
			\end{tabu}
		}
	\end{center}
	\caption{The various cases of tight $SMAs$ and their corresponding lemmata.}
	\label{Tight array cases}
	\end{figure}

\end{proof}

\section{Signed magic squares}\label{SMS}

We now turn our attention to that of a signed magic square. As in the case of a tight array, we will split the problem into several cases, which will be handled independently.
Recall that we use the notation $SMS(n;t)$ for a signed magic square with $t$ filled cells in each row and
$t$ filled cells in each column. An $SMS(n;t)$ is called
$k$-{\em diagonal} if its entries all belong to $k$ consecutive diagonals. In the case where $k = t$, we abbreviate this to simply {\em diagonal}.

We first cover the trivial case where $t < 3$. The proof of the following result is straightforward.

\begin{thm}\label{12inn}
	There exists no $SMS(n;t)$ for $t < 3$ apart from the trivial $SMS(1;1)$ containing a single zero.
\end{thm}

Now, we will split into cases based on the parities of $t$ and $n$.

\subsection{Signed magic squares with $n, t$ both odd}
In this case, the proof is quite complex, so we begin with two lemmata.

\begin{lem}
	\label{oddinoddcols}
	Let $n \geq t \geq 3$ be odd integers. Then there exists a partition $\mathcal{A}$ of the set $S=[-\frac{nt-1}{2},\frac{nt-1}{2}]$ such that every set in the collection $\mathcal{A}$ contains exactly $t$ elements and sums to zero.
\end{lem}

\begin{proof}
	We will let $n$ be an arbitrary odd integer greater than or equal to $3$ and induct on $t$.
For the base case $t=3,$ we demonstrate the existence of a partition $\mathcal{A}$ by giving it an explicit construction. Let $\mathcal{A}=\lbrace C_{c} \rbrace,$ where
$$C_{c}=\left\lbrace (c-1)-\left(\frac{3n-1}{2}\right), -\left(\frac{n-1}{2}\right)+x_{c},-(c-1)+\left(\frac{3n-1}{2}\right)-x_{c}+\frac{n-1}{2} \right\rbrace,$$ $$c \in \left\lbrace 1,2,\ldots, n \right\rbrace$$
and $x_{c}$ is defined by  $x_{c} \equiv 2\left(\frac{3n-1}{2}-(c-1)\right)-1 \pmod n, x_{c} \in \lbrace 0,1,\ldots,n-1 \rbrace.$
	
	By construction, the sum of the elements in each set $C_{c}$ is zero. Therefore, it is sufficient to show that $\mathcal{A}$ is a partition of $S.$ First, we show that $|C_{c}|=3$. Suppose to the contrary that $|C_{c}|<3$. We have cases to consider.
	
	$\mathbf{Case \hspace{0.15 cm} 1}: (c-1)-\left(\frac{3n-1}{2}\right)=-\left(\frac{n-1}{2}\right)+x_{c}$. In this case, note that $(c-1)-\left(\frac{3n-1}{2}\right) \leq -\left(\frac{n-1}{2}\right)-1$, while $-\left(\frac{n-1}{2}\right)+x_{c} \geq -\left(\frac{n-1}{2}\right)$, showing that this case is impossible.
	
	$\mathbf{Case \hspace{0.15 cm} 2}: (c-1)-\left(\frac{3n-1}{2}\right)=-(c-1)+\left(\frac{3n-1}{2}\right)-x_{c}+\frac{n-1}{2}$. This equality implies $2\left(\frac{3n-1}{2}-(c-1)\right)=x_{c}-\left(\frac{n-1}{2}\right)$. Note $2\left(\frac{3n-1}{2}-(c-1)\right) \geq n+1$, while $x_{c}-\left(\frac{n-1}{2}\right) \leq \frac{n-1}{2}$, showing that this case is impossible.
	
	$\mathbf{Case \hspace{0.15 cm} 3}: -\left(\frac{n-1}{2}\right)+x_{c}=-(c-1)+\left(\frac{3n-1}{2}\right)-x_{c}+\frac{n-1}{2}$. This equality implies $2x_{c}=\left(\frac{3n-1}{2}\right)-(c-1)+(n-1)$. Note that when $c=n$, $x_{c}=0$, so $x_{c}<\left(\frac{3n-1}{2}\right)-(c-1)=\left(\frac{n+1}{2}\right)$ in this case. Note that as $c$ decreases in increments of 1, $x_{c}$ increases in increments of $2$ until $c=\frac{n-1}{2}$, at which point $x_{c}=1$. It is straightforward to verify that $x_{c}<\left(\frac{3n-1}{2}\right)-(c-1)$ for all $c>\frac{n-1}{2}$, and obviously $x_{c}<\left(\frac{3n-1}{2}\right)-(c-1)$ for all $c \leq\frac{n-1}{2}$ as well, since $\left(\frac{3n-1}{2}\right)-(c-1)>n$ in that case. Therefore, since $x_{c}<\left(\frac{3n-1}{2}\right)-(c-1)$ and $x_{c} \leq n-1$, it follows that the equality $2x_{c}=\left(\frac{3n-1}{2}\right)-(c-1)+(n-1)$ cannot hold, as required.
	
	It remains to show that if $j \neq k$, then $C_{j} \cap C_{k}=\emptyset$. So suppose by way of contradiction that there exist $C_{j},C_{k}$ with $j \neq k$ and $C_{j} \cap C_{k} \neq \emptyset$. We again have cases to consider.
	
	$\mathbf{Case \hspace{0.15 cm} 1'}: (j-1)-\left(\frac{3n-1}{2}\right)=(k-1)-\left(\frac{3n-1}{2}\right)$, then $j=k$, a contradiction.
	
	$\mathbf{Case \hspace{0.15 cm} 2'}: (j-1)-\left(\frac{3n-1}{2}\right)=-\left(\frac{n-1}{2}\right)+x_{k}$; see Case 1 above for the same argument.
	
	$\mathbf{Case \hspace{0.15 cm} 3'}: (j-1)-\left(\frac{3n-1}{2}\right)=-(k-1)+\left(\frac{3n-1}{2}\right)-x_{k}+\frac{n-1}{2}$; see Case 2 above for the same argument.
	
	$\mathbf{Case \hspace{0.15 cm} 4'}: -\left(\frac{n-1}{2}\right)+x_{j}=-\left(\frac{n-1}{2}\right)+x_{k}$. This implies $x_{j}=x_{k}$, which means that $ 2(\frac{3n-1}{2}-(j-1))-1 \equiv 2(\frac{3n-1}{2}-(k-1))-1 \pmod n$. This implies $2j \equiv 2k \pmod n$. Since $n$ is odd, this congruence holds if and only if $j \equiv k \pmod n$, which is impossible.
	
	$\mathbf{Case \hspace{0.15 cm} 5'}: -\left(\frac{n-1}{2}\right)+x_{j}=-(k-1)+\left(\frac{3n-1}{2}\right)-x_{k}+\frac{n-1}{2}$, refer to Case 3 in the previous part of the proof.
	
	$\mathbf{Case \hspace{0.15 cm} 6'}:
-(j-1)+\left(\frac{3n-1}{2}\right)-x_{j}+\frac{n-1}{2}=-(k-1)+\left(\frac{3n-1}{2}\right)-x_{k}+\frac{n-1}{2}$. This equality implies $-(j-1)-\left (2(\frac{3n-1}{2}-(j-1))-1 \right)+(k-1)+\left (2(\frac{3n-1}{2}-(k-1))-1 \right) \equiv 0 \pmod n$, or equivalently $j \equiv k \pmod n$, which is impossible.
We conclude that $\mathcal{A}$ is a partition of $S$, so in particular the base case for $t=3$ holds.
	
	We now partition the set $S=[-\frac{3n-1}{2},\, \frac{3n-1}{2}]$ into three blocks of $n$ consecutive integers. In this case, call the three blocks $D_{1},D_{2},D_{3}$, where $D_{1}= \left[ -\left(\frac{3n-1}{2}\right), -\left(\frac{n+1}{2}\right)\right] $ and $D_{2}$ and $D_{3}$ are defined in the obvious manner. We claim that the two partitions $\mathcal{A}$ and $\{D_1,D_2,D_3\}$ of the set $S$ are orthogonal.
Notice that the first $n$ consecutive integers in $S$ appear in the sets $C_1,C_2,\ldots,C_{n}$, respectively.
The next $n$ consecutive integers appear in the sets $C_{n},C_{\frac{n-1}{2}},\ldots,C_{\frac{n+1}{2}},$ respectively, where for each consecutive integer the index on $C_c$ is increased by $\frac{n-1}{2}$ and the result is taken modulo $n$ (here, the residues used are $\lbrace 1,2,\dots,n \rbrace$). To see why this is true, note that when $c=n$, $x_c=0$, and $-\left(\frac{n-1}{2}\right)+x_c=-\left(\frac{n-1}{2}\right)$.  Moreover, when the index on $C_c$ is increased by $\frac{n-1}{2}$ $\pmod n$, a quick computation reveals that $x_c$ is increased by one, so $-\left(\frac{n-1}{2}\right)+x_c$ is increased by one. Since $\frac{n-1}{2}$ is coprime to $n$, all the sets in the partition will have exactly one representative element in this block of $n$ integers. Finally, the last $n$ consecutive integers appear in the set $C_{\frac{n+1}{2}}, C_{\frac{n+1}{2}+1}, \ldots, C_{\frac{n-1}{2}},$ respectively, where for each consecutive integer the index on $C_c$ is increased by $1$ and the result is taken modulo $n$. To see why this is true, note that when $c=\frac{n+1}{2}$, $-(c-1)+\left(\frac{3n-1}{2}\right)-x_{c}+\frac{n-1}{2}=\frac{n+1}{2}$. Moreover, when $c$ increases by one modulo $n$, $x_c$ decreases by $2$ modulo $n$, so  $-(c-1)+\left(\frac{3n-1}{2}\right)-x_{c}+\frac{n-1}{2}$ increases by $1$. The only break occurs when going from $c=n$ and $x_c=0$ to $c=1$ and $x_c=n-2$, in which case the quantity $-(c-1)+\left(\frac{3n-1}{2}\right)-x_{c}+\frac{n-1}{2}$ still increases by 1.
Hence, each block of $n$ consecutive numbers in the set $S=[-\frac{3n-1}{2},\, \frac{3n-1}{2}]$
contains a representative from every set in $\mathcal{A}$.

   We will refer to these blocks $D_1,D_2,D_3$ as ``shiftable'' for the following reason. Suppose we add
integers $\alpha_1,\alpha_2$ and $\alpha_3$ to all of the elements in the sets $D_1$, $D_2$ and $D_3$, respectively. Further, let $\alpha_1,\alpha_2,\alpha_3$ be chosen such that the shifted blocks $D_1',D_2',D_3'$ are disjoint. This shift naturally induces new disjoint sets $C'_1,\ldots, C'_{n}$, where $C'_c=\left\lbrace (c-1)-\left(\frac{3n-1}{2}\right)+\alpha_1, -\left(\frac{n-1}{2}\right)+x_{c}+\alpha_2,-(c-1)+\left(\frac{3n-1}{2}\right)-x_{c}+\frac{n-1}{2}+\alpha_3 \right\rbrace$. Then it is clear that $\sum_{x \in C'_c}x=\alpha_1+\alpha_2+\alpha_3$, and this constant is independent of $c$.
	
	For the inductive step, suppose that there exists a partition of $\mathcal{A}$ of the set
$S=[-\frac{nt-1}{2},\, \frac{nt-1}{2}]$
into sets $C_{1},\ldots,C_{n}$, where sets of $n$ consecutive integers form shiftable blocks $D_{1},\ldots, D_{t}$. We wish to show that there exists a partition $\mathcal{A}'$ of the set
$S'=[-\frac{n(t+2)-1}{2},\, \frac{n(t+2)-1}{2}]$ into sets $C'_{1},\ldots,C'_{n}$ where sets of $n$ consecutive integers form shiftable blocks $D'_{1},\ldots, D'_{t+2}$. To begin with, shift the blocks $D_{1},\ldots,D_{t-1}$ by subtracting $n$ from each of the elements. Then, shift the block $D_t$ by adding $n$ to each of the elements. These shifts induce new disjoint sets $C'_{1},\ldots,C'_{n}$, each of which have $t$ elements. By construction, the elements in each of the $C'_{c}$ sum to $-n(t-2)$. We wish to "complete" the partition by adding two elements from the set $B=\{(n(t-4)+1)/2,(n(t-4)+1)/2+1,\ldots, (nt-1)/2\}$ to each $C'_{c}$. Place the first $n$ consecutive integers in $B$ into $C'_{\frac{n+1}{2}},C'_{1},\ldots,C'_{n}$, respectively, and the next $n$ consecutive integers in $B$ into $C'_{n},C'_{\frac{n-1}{2}},\ldots,C'_{\frac{n+1}{2}}$, respectively. Note that these patterns are mirror images of each other, and that while in the first case the index on the $C_c$ increases by $\frac{n+1}{2}$ for each consecutive integer, in the second case it increases by $\frac{n-1}{2}$, both of course taken modulo $n$ using residues $\lbrace 1,\dots,n \rbrace$. It is routine to verify that this method ensures that the quantity $n(t-2)$  is added to each of the $C_c'$s. By construction, the elements in each $C'_{c}$ will sum to 0, and we have constructed a partition $\mathcal{A}'$ of the set $S'=[-\frac{n(t+2)-1}{2},\, \frac{n(t+2)-1}{2}]$
 into sets $C'_{1},\ldots,C'_{n}$ of cardinality $t+2$ that all sum to 0. Note that this method also guarantees that we end up with a partition with the same property of shiftable blocks. In particular, it is easy to see in general this partitioning method involves placing the first $n$ consecutive integers in the order $C_1,C_2 \ldots C_n$, placing the next $n$ in the order $C_{n},C_{\frac{n-1}{2}},\ldots,C_{\frac{n+1}{2}}$, and placing the following $n$ in reverse order. This pattern of skipping by $\frac{n-1}{2}$ and then skipping by $\frac{n+1}{2}$ continues to alternate until the last $n$ consecutive integers, which are placed $C_{\frac{n+1}{2}},C_{\frac{n+1}{2}+1},\ldots, C_{\frac{n-1}{2}}$ where for each consecutive integer the index on $C_c$ is increased by $1$ and the result is taken modulo $n$.
\end{proof}

\begin{lem}\label{oddinoddrows}
	Let $n \geq t \geq 3$ be odd integers, and let $\{D_i\}$ and $\{C_c\}$ be defined as in the proof of Lemma \ref{oddinoddcols}. Then there exists another partition $\{R_r\}$ of the set $S=[-\frac{nt-1}{2},\frac{nt-1}{2}]$ orthogonal to $\{C_c\}$ such that each $R_r$ has $t$ entries and sums to zero.
\end{lem}

\begin{proof}
	Let $c_{i, j}$ be such that $C_{c_{i, j}}$ contains the $j$th entry, in ascending order, within $D_i$. Then thus far, we have partitioned $C_c$ such that that there exist $k_i$ and $b_i$ such that $c_{i, j} \equiv k_i j + b_i \pmod n$. Also, $k_i = 1$ if $i = 1 \text{ or } t$, $k_i = \frac{n-1}{2}$ if $i$ is even, and $k_i = \frac{n+1}{2}$ otherwise. Lastly, $b_i = n$ if $i \neq t$ is odd, $b_i = \frac{n+1}{2}$ if $i$ is even, and $b_t = \frac{n-1}{2}$.
	
	Let $j_{i, c}$ be such that $c_{i, j_{i, c}} = c$. Then, solving the equation $c \equiv k_i j_{i, c} + b_i \pmod n$ gives us $k_i^{-1} c - k_i^{-1} b_i \equiv j_{i, c} \pmod n$. We may then determine values for the coefficient and constant in this congruence via modular algebra:
	
	$$
	k_i^{-1} \equiv k'_i = \begin{cases}
	1 & \text{if } i = 1 \text{ or } t \\
	-2 & \text{if } i \text{ is even} \\
	2 & \text{otherwise.}
	\end{cases}
	$$
	
	$$
	- k_i^{-1} b_i \equiv b'_i = \begin{cases}
	0 & \text{if } i \neq t \text{ is odd} \\
	1 & \text{if } i \text{ is even} \\
	\frac{n+1}{2} & \text{if } i = t.
	\end{cases}
	$$
	
	For convenience, we will let $j_{i, c + n} = j_{i, c}$ for all $c$ and introduce some definitions.
Given $t, n, i$, a {\em break} is a value of $c$ in $[1, n]$ such that $j_{i, c + 1} - j_{i, c} \neq k'_i$ (note that as determined above, $j_{i, c + 1} - j_{i, c} \equiv k'_i \pmod n$). A break is {\em positive} if $j_{i, c + 1} - j_{i, c} > k'_i$ and {\em negative} otherwise. The {\em magnitude} of a break is $|j_{i, c + 1} - j_{i, c} - k'_i|$, and the {\em signed magnitude} is $j_{i, c + 1} - j_{i, c} - k'_i$. For example,
let $t = 5, n = 7, i = 2$. Then $j_{i, c}$ for $c \in [1, 7]$ are the values $6, 4, 2, 7, 5, 3, 1$. (These are congruent modulo 7 to the values $-2c + 1$.) There are two breaks: one at $c = 3$ where $j$ goes from 2 to 7, and one at $c = 7$ where $j$ goes from 1 to 6. Both are positive breaks, as $7 - 2 = 6 - 1 > -2$.

We next determine where exactly the breaks occur given $t, n, i$, and what their signs and magnitudes are.
	
	If $i = 1$, then $j_{1, c} \equiv c \pmod n$. As $j_{1, c} \in [1, n]$ it follows that $j_{1, c} = c$ for $c \in [1, n]$. This sequence has one negative break of magnitude $n$ at $c = n$, where $j_{1, c + 1} - j_{1, c} = 1 - n = k'_1 - n$.
	
	If $i$ is even, then $j_{i, c} \equiv -2c + 1 \pmod n$. This sequence has two positive breaks of magnitude $n$. One occurs at $c = n$, where $j_{i, c + 1} - j_{i, c} = (n - 1) - 1 = n - 2 = k'_i + n$. The other occurs at $c = \frac{n-1}{2}$, where $j_{i, c + 1} - j_{i, c} = n - 2 = k'_i + n$.
	
	If $i = t$, then $j_{t, c} \equiv c + \frac{n+1}{2} \pmod n$. This sequence has one negative break of magnitude $n$ at $c = \frac{n-1}{2}$, where $j_{t, c + 1} - j_{t, c} = 1 - n = k'_t - n$.
	
	If $i \neq 1$, $i \neq t$, and $i$ is odd, then $j_{i, c} \equiv 2c \pmod n$. This sequence has two negative breaks of magnitude $n$. One occurs at $c = n$, where $j_{i, c + 1} - j_{i, c} = 2 - n = k'_t - n$. The other occurs at $c = \frac{n-1}{2}$, where $j_{i, c + 1} - j_{i, c} = 1 - (n - 1) = 2 - n = k'_t - n$.
	
	Now we  define  $R_r$. For $r \in [1, n]$, we define $R_r$ as follows: for $i \in [1, t]$, $R_r$ contains the single value in both $D_i$ and in $C_c$ where $c \equiv r + \frac{n-1}{2}(i - 1) \pmod n$. For clarity, we define the function $s$ as $s(i) = \frac{n-1}{2}(i - 1)$ and use $s(i)$ from now on.
	
	Next we  define  an analogous symbol $j'$ to $j$: $j'_{i, r}$ is the value in $[1, n]$ such that $R_r$ contains the $j'_{i, r}$th entry, in ascending order, within $D_i$. Note that this value must also be in $C_c$ such that $j_{i, c} = j'_{i, r}$. It follows that $j'_{i, r} = j_{i, r + s(i)}$.
As before, for convenience, we will let $j'_{i, r + n} = j'_{i, r}$ for all $r$.
	
	We may define breaks in $j'$ in the same manner as for breaks in $j$. Let $\beta_{i, r}$ be the signed magnitude of the break that occurs between $j'_{i, r}$ and $j'_{i, r + 1}$ if one exists and 0 otherwise. In other words, $\beta_{i, r} = j'_{i, r + 1} - j'_{i, r} - k'_i$. We again divide into cases based on $i$.
	
	If $i = 1$, then $j'_{1, r} = j_{1, c}$ because $s(1) = 0$. So $\beta_{1, r} = -n$ if $r = n$ and 0 otherwise. Note that we can alter this solution using the modular congruence modulo $n$ to be $r = 0 = s(1) = s(i)$.
	
	If $i$ is even, then $\beta_{i, r} = n$ when $r + s(i) = n$ or $r + s(i) = \frac{n-1}{2}$. Solving for $r$, we obtain $r = n - s(i)$ or $r = \frac{n-1}{2} - s(i)$. Simplifying and using the congruence modulo $n$, we have $r = -s(i)$ and $r = -s(i - 1)$.
	
	If $i = t$, then $\beta_{i, r} = -n$ when $r + s(i) = \frac{n-1}{2}$. As before, this is equivalent to $r = -s(i - 1)$.
	
	If $i \neq 1$, $i \neq t$, and $i$ is odd, then $\beta_{i, r} = -n$ when $r + s(i) = n$ or $r + s(i) = \frac{n-1}{2}$. As before, this is equivalent to $r = -s(i)$ or $r = -s(i - 1)$.
	
	Let $r$ be chosen arbitrarily, and let $z \in [1, t-1]$ be such that $r \equiv -s(z)$. Then $\beta_{z, r}$ and $\beta_{z+1, r}$ are nonzero and are each other's opposites. Also, this matching covers and partitions all nonzero values of $\beta_{i, r}$. Therefore, $\sum_{i=1}^t \beta_{i, r} = 0$.
Note also that $\sum_{i=1}^t k'_i = -1 + 2 + -2 + 2 + -2 + \dots + 2 + -1 = 0$.
	
	Now let $\delta_i$ be the offset of $D_i$ such that the $j$th element of $D_i$ in ascending order is $\delta_i + j$. Then it follows by definition that $\sum R_r = \sum_{i=1}^t (\delta_i + j'_{i, r})$. We will consider the difference between two consecutive $R_r$:
$$\begin{array}{lcl}	
	\sum R_{r+1} - \sum R_r
	&=& \sum_{i=1}^t (\delta_i + j'_{i, r + 1}) - \sum_{i=1}^t (\delta_i + j'_{i, r})\\
	&=& \sum_{i=1}^t (\delta_i + j'_{i, r + 1} - \delta_i - j'_{i, r})
	= \sum_{i=1}^t (j'_{i, r + 1} - j'_{i, r})\\
	&=& \sum_{i=1}^t (k'_i + \beta_{i, r})
	= \sum_{i=1}^t k'_i + \sum_{i=1}^t \beta_{i, r} = 0.
\end{array}$$
	
	So any $R_r$ has the same sum as $R_{r+1}$; by induction, all $R_r$ have the same sum. As the sum of these sums is $\sum_{i=-\frac{tn-1}{2}}^{\frac{tn-1}{2}} i = 0$, each sum individually is also zero.
\end{proof}

\begin{thm}\label{oddinodd}
	Let $n \geq t \geq 3$ be odd integers. Then there exists an $SMS(n;t)$.
\end{thm}

\begin{proof}
	Let $\{C_c\}$ be defined as in Lemma \ref{oddinoddcols} and $\{R_r\}$ as in Lemma \ref{oddinoddrows}. Let $B=[b_{r,c}]$, where $b_{r, c}$ is the single element in $R_r$ and $C_c$ if they have nonempty intersection and is left blank if they do not. Then every row and every column of $B$ sums to zero, and $B$ has exactly $t$ entries in each of its $n$ rows and $n$ columns. Thus $B$ is the desired array (see Figure \ref{5in7}).
\end{proof}

	\begin{figure}[ht]
		$$\begin{array}{| c | c | c | c | c | c | c |}
		\hline
		\cellcolor{gray!25}-17 & -16 & -15 & -14 & -13 & -12 & -11 \\ \hline
		-5 & -7 & -9 & \cellcolor{gray!25}-4 & -6 & -8 & -10 \\ \hline
		-2 & 0 & 2 & -3 & -1 & 1 & \cellcolor{gray!25}3 \\ \hline
		9 & 7 & \cellcolor{gray!25}5 & 10 & 8 & 6 & 4 \\ \hline
		15 & 16 & 17 & 11 & 12 & \cellcolor{gray!25}13 & 14 \\ \hline
		\end{array}$$
		
		$$\begin{array}{| c | c | c | c | c | c | c |}
		\hline
		-17 & & 5 & -4 & & 13 & 3 \\ \hline
		-2 & -16 & & 10 & -6 & & 14 \\ \hline
		15 & 0 & -15 & & 8 & -8 & \\ \hline
		& 16 & 2 & -14 & & 6 & -10 \\ \hline
		-5 & & 17 & -3 & -13 & & 4 \\ \hline
		9 & -7 & & 11 & -1 & -12 & \\ \hline
		& 7 & -9 & & 12 & 1 & -11 \\ \hline
		\end{array}$$
		
		\caption{Top: an array whose rows are $\{D_i\}$ and whose columns are $\{C_c\}$, for $n = 7, t = 5$, highlighting the elements of $R_1$. Bottom: the $SMS(7,5)$ given by Theorem \ref{oddinodd}.}
		\label{5in7}
	\end{figure}

\subsection{Signed magic squares with $n$ odd, $t$ even}

For most of the remaining square cases, we will need the following two lemmata.

\begin{lem}
	\label{4inndiag}
	For all positive integers $n \geq 4$, there exists a shiftable diagonal $SMS(n;4)$.
\end{lem}
	
\begin{proof}
	Let $n$ be a positive integer. Define array $A=[a_{i,j}]$ by:
	$$a_{i,j}= \begin{cases}
	i & i=j,\;1\leq j<n \\
	-n & i=j=n \\
	-i & i=j-1,\; 2\leq j <n-1 \\
	n+1 & i=n-1,\; j=n \\
	2n-i & i=j-2,\; 3 \leq j \leq n\\
	-(2n-i) & i=j-3,\; 4\leq j \leq n\\
	-(n+2) & i=n-2,\; j=1\\
	-(n-1) & i=n-1,\; j=1\\
	2n & i=n,\; j=1\\
	-(n+1) & i=n-1,\; j=2\\
	n & i=n,\; j=2\\
	-2n & i=n,\; j=3
	\end{cases}$$
	and the cells are left empty otherwise. Figure \ref{4in8diag} provides an example of such an array for $n=8$.
	
	It is straightforward to verify that $A$ contains all the entries in the set $[-2n, 2n]$ and that each of the entries appears exactly one time in the array. Moreover, by construction this array has the property that four adjacent diagonals are filled. It remains to check that rows and columns sum to 0.

	For the rows, if $1 \leq i \leq n-3$, the sum of the entries in row $i$ is $i-i+(2n-i)-(2n-i)=0$, as desired. If $i=n-2$, the sum is $-(n+2)+(n-2)-(n-2)+(2n-(n-2))=0$. If $i=n-1$, the sum is $-(n-1)+(n-1)-(n-1)+(n-1)=0$. Finally, the sum of the entries in the last row must be $0$, because the sum of all the entries in the array is $0$.

	For the columns, if $j=1$ then the sum is $1-(n+2)-(n-1)+2n=0$. If $j=2$ the sum is $-1+2-(n+1)+n=0$. If $j=3$, the sum is $(2n-1)-2+3-2n=0$. If $4 \leq j \leq n-1$, the sum is $j-(j-1)+(2n-(j-2))-(2n-(j-3))=0$, as required. Finally, the sum of the entries in the last column must be $0$, because the sum of all the entries in the array is $0$. This completes the proof.
\end{proof}

	\begin{figure}[ht]
		$$\begin{array}{|c|c|c|c|c|c|c|c|}
		\hline
		1 & -1 & 15 & -15 & & & & \\
		\hline
		& 2 & -2 & 14 & -14 & & & \\
		\hline
		& & 3 & -3 & 13 & -13 & & \\
		\hline
		& & &  4 & -4 & 12 & -12 & \\
		\hline
		& & & & 5 & -5 & 11 & -11 \\
		\hline
		-10 & & & & & 6 & -6 & 10 \\
		\hline
		-7 & -9 & & & & & 7 & 9 \\
		\hline
		16 & 8 &-16 & & & & & -8 \\
		\hline
		\end{array}$$
		\caption{The shiftable diagonal $SMS(8;4)$ given by Lemma \ref{4inndiag}.}
		\label{4in8diag}
	\end{figure}

\begin{lem} \label{add4tot}
	Assume that there exists a $k$-diagonal $SMS(n;t)$ with $k \leq n - 4$ and either $t$ or $n$ is even. Then there exists a $(k+4)$-diagonal $SMS(n;t+4)$.
\end{lem}

\begin{proof}
	Let $A$ be a $k$-diagonal $SMS(n;t)$ with $k \leq n - 4$ and either $t$ or $n$ even. Note that the entries in $A$ can be partitioned into $n$ diagonals, and that as $k \leq n - 4$, at least four consecutive diagonals are empty, and we may choose these four consecutive diagonals to be adjacent to the $k$ diagonals in which $A$'s entries are contained.
	
	Let $B$ be the shiftable diagonal $SMS(n;4)$ given by Lemma \ref{4inndiag}. Let $B'$ be a copy of $B$ with the entries shifted to have absolute values in $[\frac{tn}{2} + 1, \frac{(t+4)n}{2}]$ rather than $[1, 2n]$ and with the columns permuted to place the four diagonals of $B$ into the same cells as the four empty diagonals of $A$.
	
	Then $A$ and $B'$ do not share any filled cells, and together their entries occupy $k+4$ consecutive diagonals; each array has zero row and column sums; and the two arrays together use each number in $[-\frac{(t+4)n}{2}, -1]$ and $[1, \frac{(t+4)n}{2}]$ exactly once. By combining the two arrays into one, $A + B'$, we achieve the desired signed magic square.
\end{proof}

Now, we move on to actually considering the case where $t$ is even and $n$ is odd. We will prove this by an induction, one of whose base cases is complex enough to warrant another lemma.

\begin{lem}\label{6inodddiag}
	Let $n > 6$ be odd. Then there exists a shiftable diagonal $SMS(n;6)$.
\end{lem}

\begin{proof}
	From Lemma \ref{oddxodd}, there exists an $SMS(3,n)$, say $A$, using entries in $[-\frac{3n-1}{2}, \frac{3n-1}{2}]$. Let $a'_{i, j} = a_{i, j} + \frac{3n+1}{2}$. Then $A'$ uses each number in $[1, 3n]$ exactly once, and the sum of every column of $A'$ is the same.
	
	We define an array $B=[b_{i,j}]$ as follows. For $i \in [1, 3]$ and $j \in [1, n]$, let $i', j' \in [1, n]$ with $i' \equiv 2i - 1 + j - 1 \pmod n$ and $j' = j$. Then $b_{i', j'} = a'_{i, j}$, and $b_{i', j' + 1} = -a'_{i, j}$ (using the convention that $b_{i', n + 1} = b_{i', 1}$). All other cells in $B$ are left empty.
	
	We will now determine the possible values $i'$ and $j'$ can take for $b_{i', j'}$ to be filled. We must have $i' \equiv 2i + j - 2 \pmod n$ and $j' = j + s$ for $i \in [1, 3]$, $j \in [1, n]$, and $s \in \{0, 1\}$.
	
	Note that these equations imply $i' - j' \equiv 2i + j - 2 - j - s \equiv 2i - s - 2 \pmod n$; in fact, this is equivalent to the above conditions, because $j$ can range from $1$ to $n$. Therefore $i' - j'$ must be one of the values $\{2i - s - 2\} = \{-1, 0, 1, 2, 3, 4\}$. Thus, $B$ consists of six consecutive diagonals.
	
	Observe that in row $i'$ of $B$, as we have defined $B$, the nonempty cells partition naturally into pairs of opposite entries; so the row sums to zero.
	
	Now in column $j'$, there are six entries. These six entries are $a_{1, j'}$, $a_{2, j'}$, $a_{3, j'}$, $-a_{1, j' - 1}$, $-a_{2, j' - 1}$, and $-a_{3, j' - 1}$ (using the convention that $a_{i, 0} = a_{i, n}$). One may verify using the congruences given above that these six numbers are indeed placed in column $j'$. Then the sum of column $j'$ is the sum of these six entries. As stated earlier, $a_{1, j'} + a_{2, j'} + a_{3, j'} = a_{1, j' - 1} + a_{2, j' - 1} + a_{3, j' - 1}$, so the column sums to zero. Thus array $B$ has the desired properties (see Figure \ref{6in7diag}).
\end{proof}
	
	\begin{figure}[ht]
		$$\begin{array}{| c | c | c | c | c | c | c |}
		\hline
		1& 20 & 9 & 14 & 5 & 10 & 18 \\ \hline
		19 & 11 & 3 & 15 & 16 & 6 & 7 \\ \hline
		13 & 2 & 21 & 4 & 12 & 17 & 8 \\ \hline\hline
		1 & -1 & & 4 & -4 & 6 & -6 \\ \hline
		-7 & 20 & -20 & & 12 & -12 & 7 \\ \hline
		19 & -19 & 9 & -9 & & 17 & -17 \\ \hline
		-8 & 11 & -11 & 14 & -14 & & 8 \\ \hline
		13 & -13 & 3 & -3 & 5 & -5 & \\ \hline
		& 2 & -2 & 15 & -15 & 10 & -10 \\ \hline
		-18 & & 21 & -21 & 16 & -16 & 18 \\ \hline
		\end{array}$$
		
		\caption{Top: an $SMS(3,n)$, shifted to use values from $[1, 21]$. Bottom: the corresponding
          $SMS(7;6)$ given by Lemma \ref{6inodddiag}.}
		\label{6in7diag}
	\end{figure}

Now, we give the full induction argument.

\begin{thm}
	\label{eveninodd}
	Given $n > t > 3$ with $n$ odd and $t$ even, there exists a shiftable diagonal $SMS(n;t)$.
\end{thm}
\begin{proof}
	We prove this by induction with two base cases, the cases where $t = 4$ and $t = 6$, which are given by Lemma \ref{4inndiag} and by Lemma \ref{6inodddiag}, respectively.
	
	Now assume that such arrays exist for all even $t' < t$ greater than $3$, in particular for $t' = t - 4$. Then we may apply Lemma \ref{add4tot} to construct a diagonal $SMS(n;t)$. Hence, the statement is true by induction.
\end{proof}

\subsection{Signed magic squares with $n,t$ both even}

We begin with the case $t$ or $n$ is a multiple of $4$.

\begin{lem} \label{tnmult8}
	Let $t$ and $n$ be even positive integers, with $t \leq n$ and either $t$ or $n$ divisible by $4$. Then there exists an $SMS(n;t)$.
\end{lem}

\begin{proof}
	By the assumption $\frac{nt}{2}$ is a multiple of four. By Theorem \ref{tightHeffter}, this implies that there exists a tight integer $\frac{n}{2} \times t$ Heffter array, say $A=[a_{i,j}]$. Let $A_i$ denote the row $i$ of $A$.
	We will create two orthogonal partitions, $\{R_r\}$ and $\{C_c\}$, of the set $X = [-\frac{nt}{2}, -1] \cup [1, \frac{nt}{2}]$. Note that due to the definition of Heffter array, $X$ is precisely the set of entries in $A$ with their opposites.
Therefore we will let $R_{2i - 1} = A_i$, and let $R_{2i}$ contain the opposites of $A_i$, for $i \in [1, \frac{n}{2}]$. It is apparent that $E_1=\{R_r\mid r \mbox{ is odd}\}$ partition the set of entries in $A$ while
 $E_2=\{R_r\mid r \mbox { is even}\}$ partition the set of their opposites; therefore $E_1\cup E_2$ partition $X$. In addition, the sum of the entries in a given $R_r$ will either be the sum of a row of $A$ (zero, by definition), or its opposite (also zero). Lastly, we note that $|R_r| = t$, the cardinality of a row of $A$.
	
	Now we  define  $C_c$. For $c \in [1, \frac{n}{2}]$ and $j \in [1, \frac{t}{2}]$, let $i \in [1, \frac{n}{2}]$ be such that $i - j + 1 \equiv c \pmod {\frac{n}{2}}$. Then $-a_{i, j},a_{i,j} \in C_c$. For $c \in [\frac{n}{2} + 1, n]$ and $j \in [\frac{t}{2} + 1, t]$, let $i \in [1, \frac{n}{2}]$ be such that $i - j + 1 + \frac{t}{2} + \frac{n}{2} \equiv c \pmod {\frac{n}{2}}$. Then again $-a_{i, j},a_{i,j} \in C_c$.
	
	Note that exactly one $C_c$ contains each entry in $A$, as given $j$ and $c$, where $j \leq \frac{t}{2}$ iff $c \leq \frac{n}{2}$, there will be exactly one $i$ solving the modular congruence above. As $C_c$ contains the opposites to its elements as well, it follows that $\{C_c\}$ partitions $X$. Also, note that $C_c$ contains two elements from each of half of the columns of $A$, so $|C_c| = t$.
	
	Now by definition of $C_c$, if $C_c$ contains $x$, then $C_c$ contains $-x$; so the sum of each $C_c$ is zero. Lastly, we need to prove that $\{R_r\}$ and $\{C_c\}$ are orthogonal partitions. Let $i \in [1, \frac{n}{2}]$ be arbitrary. Then $R_{2i - 1} = A_i$. Assume that there is some $c$ such that $C_c$ and $A_i$ have two elements, namely
$a_{i, j}$ and $a_{i, j'},$ in common. Because different $C_c$ are used for the left and right halves of $A$, we can assume that $j$ and $j'$ are on the same side, i.e. $j, j' \in [1, \frac{t}{2}]$ or $j, j' \in [\frac{t}{2} + 1, t]$. In the former case, we have $i - j + 1 \equiv i - j' + 1 \pmod {\frac{n}{2}}$; in the latter case, we have $i - j + 1 + \frac{t}{2} + \frac{n}{2} \equiv i - j' + 1 + \frac{t}{2} + \frac{n}{2} \pmod {\frac{n}{2}}$. In either case, canceling, $j \equiv j' \pmod {\frac{n}{2}}$. But this is impossible, as $|j' - j|$ is at most $\frac{t}{2} - 1$, and $\frac{t}{2} - 1 < \frac{t}{2} \leq \frac{n}{2}$.
	
	Therefore $C_c$ and $A_i = R_{2i-1}$ have at most one element in common. Because $C_c$ contains opposites to all of its elements, we may similarly say that $C_c$ and $R_{2i}$ have at most one element in common, as if they shared two elements $x$ and $x'$, $C_c$ and $A_i$ would share $-x$ and $-x'$. So $\{C_c\}$ and $\{R_r\}$ are orthogonal partitions.
	
	Define array $B$ as follows: the cell $(r,c)$ of $B$ contains the single element in $R_r\cap C_c$ if they have nonempty intersection and is left blank otherwise. Then every row and every column of $B$ sums to zero, and $B$ has exactly $t$ entries in each of its $n$ rows and $n$ columns. Thus $B$ is the desired array (see Figure \ref{4in6}).
	
	\begin{figure}[ht]
		$$\begin{array}{| c | c | c | c |}
		\hline
		1 & 2 & 3 & -6 \\ \hline
		8 & -12 & -7 & 11 \\ \hline
		-9 & 10 & 4 & -5 \\ \hline
		\end{array}$$
		
		$$\begin{array}{| c | c | c | c | c | c |}
		\hline
		1 & 2 & 3 & -6 & & \\ \hline
		-1 & -2 & -3 & 6 & & \\ \hline
		& & 8 & -12 & -7 & 11 \\ \hline
		& & -8 & 12 & 7 & -11 \\ \hline
		4 & -5 & & & -9 & 10 \\ \hline
		-4 & 5 & & & 9 & -10 \\ \hline
		\end{array}$$
		
		\caption{Top: a $3 \times 4$ tight integer Heffter array. Bottom: the corresponding $SMS(6;4)$.}
		\label{4in6}
	\end{figure}
\end{proof}

The rest of this case proceeds much as the case where $t$ is even and $n$ is odd, but with one important difference: instead of a diagonal $SMS(n;6)$ for $n$ even, we construct a $7$-diagonal $SMS(n;6)$ and show that this gives sufficient results.

\begin{lem}\label{6in2mod4diag}
	Let $n\equiv 2 \pmod 4$ and $n \geq 10$. Then there exists a shiftable $7$-diagonal $SMS(n;6)$.
\end{lem}

\begin{proof}
	In Lemma \ref{oddinodd}, we gave a construction that partitions the interval $[-\frac{3m-1}{2}, \frac{3m-1}{2}]$ into sets of three of equal sum for any odd $m \geq 3$. Let $m = n - 1$, and after carrying out this partition, add $\frac{3m-1}{2} + 1$ to every number in the partition, giving a partition of $[1, 3n-3]$ into sets of three of equal sum $\frac{9n-6}{2}$. We make three other observations about this partition.
	
	\begin{enumerate}
		\item The numbers in $[1, n-1]$ are placed in distinct sets, as are the numbers in $[n, 2n-2]$ and $[2n-1, 3n-3]$.
		\item The number $1$ is in the same set in the partition as the number $2n-3$.
		\item The number $n-1$ is in the same set in the partition as the number $n$.
	\end{enumerate}
	
	Now define the function $p$ as follows:
	
	$$p(x) = \begin{cases}
	x + 1 & \text{if } x < \frac{3n-2}{2} \\
	x + 2 & \text{if } x \geq \frac{3n-2}{2}
	\end{cases}$$
	
	Note that on the domain $[1, 3n - 3]$, the range of $p$ is $[1, 3n] \setminus \{1, \frac{3n}{2}, 3n\}$. These three remaining numbers add to $S = \frac{9n + 2}{2}$.
	
	We apply $p$ to every element of our previous partition of $[1, 3n-3]$ to partition the other elements of $[1, 3n]$ into sets of three. It follows from observation 1 above that $\frac{n-2}{2}$ of these, which we will call ``the first class,'' now have sum $\frac{9n-6}{2} + 1 + 1 + 2 = \frac{9n+2}{2}$, while the other $\frac{n}{2}$, which we call ``the second class,'' have sum $\frac{9n-6}{2} + 1 + 2 + 2 = \frac{9n+4}{2}$.
	
	Label every set in the partition plus $\{1, \frac{3n}{2}, 3n\}$ with the labels $P_1, P_2, \dots, P_n$, such that $i \in P_i$. Then from observation 2 above, $p(2n - 3) = 2n - 1 \in P_2$. This means that $P_2$ is of the second class as defined earlier, so $\sum P_2 = \frac{9n+4}{2}$. Also, note that from observation 3, $p(n) = n + 1 \in P_n$.
	
	We now pair the $P_i$ into pairs $B_j$, with $j \in [1, \frac{n}{2}]$, as follows: $B_1$ contains $P_1$ and $P_2$. Then the other $B_j$ may be chosen arbitrarily, following two constraints.
	
	\begin{enumerate}
		\item Each $B_j$, $j\neq 1$, contains two $P_i$ with equal sums. (This is possible because $\sum P_1 \neq \sum P_2$, and of the remaining $(n-2)$ sets, exactly $\frac{n-2}{2}$ are in the first class and $\frac{n-2}{2}$ are in the second class, as defined above. Because $n \equiv 2 \pmod 4$, we know that $\frac{n-2}{2}$ is even.)
		\item $P_n \in B_2$.
	\end{enumerate}
	
	Now, we define an $n \times n$ array $A=[a_{i,j}]$ as follows. Let $j \in [1, \frac{n}{2}]$ and $k \in [1, 3]$. Then $B_j$ contains two $P_i$, say $P_{i_1}$ and $P_{i_2}$, with $i_1 > i_2$. Arrange the elements in each of these in order, so that we may refer to them as $P_{i, 1}$, $P_{i, 2}$, and $P_{i, 3}$, for $i \in \{i_1, i_2\}$. We will place easily satisfied constraints on this otherwise arbitrary labeling: $P_{1, 2} = 1, P_{2, 2} = 2, P_{n, 1} = n, P_{n, 2} = n + 1$.
	
	For convenience, we will let the indices in $A$ ``wrap around,'' so that e.g. $a_{n+1, n+3} = a_{1, 3}$. Then for each $j, k$ as above, we fill the following cells:
	$$a_{2j+2k-3, 2j-1} = P_{i_1, k};
	a_{2j+2k-3, 2j} = -P_{i_1, k};
	a_{2j+2k-2, 2j-1} = -P_{i_2, k}
	a_{2j+2k-2, 2j} = P_{i_2, k}.$$
	We leave the other cells in $A$ empty.
	
	Note that in this array, each row contains three numbers and their opposites and thus sums to zero. As for the columns, each column $2j-1$ contains $P_{i_1}$ and the opposite of $P_{i_2}$, while column $2j$ contains $P_{i_2}$ and the opposite of $P_{i_1}$. It follows that every column sums to zero except for the first column, which sums to $\sum P_2 - \sum P_1 = \frac{9n+4}{2} - \frac{9n+2}{2} = 1$, and the second column, which sums to $\sum P_1 - \sum P_2 = -1$.
	
	We will make a few other observations about this array, related to which numbers are in which cells: $a_{3, 1} = P_{2, 2} = 2$, $a_{4, 2} = P_{1, 2} = 1$, $a_{3, 3} = P_{n, 1} = n$, $a_{3, 4} = -n$, $a_{5, 3} = P_{n, 2} = n+1$, $a_{5, 4} = -(n+1)$, $a_{4, 3} + a_{4, 4} = 0$.
	
	Let $A'$ be defined as $A$ with several exceptions:	
		$a'_{3, 1} = a_{4, 2}$;
		$a'_{4, 2} = a_{3, 1}$;
		$a'_{3, 3} = a_{5, 3}$;
		$a'_{4, 3} = a_{3, 3}$;
		$a'_{5, 3} = a_{4, 3}$;
		$a'_{4, 4} = a_{5, 4}$;
		$a'_{5, 4} = a_{4, 4}$.

	We see that this is a permutation, so the (multi-)set of entries used in $A'$ is the same as that of $A$: the elements $[-3n, -1]$ and $[1, 3n]$, each exactly once.
	
	In $A'$, every element is in the same column as in $A$ with the exception of $a'_{3, 1} = 1$ and $a'_{4, 2} = 2$. It follows that the sum of column $1$ is now $1 - 2 + 1 = 0$, while the sum of column $2$ is now $-1 - 1 + 2 = 0$.
	
	The only rows that have changed from $A$ to $A'$ are rows $3$, $4$, and $5$.
$$\begin{array}{lcl}
    \sum_c a'_{3, c} &= &\sum_c a_{3, c} - a_{3, 3} + a_{5, 3} - a_{3, 1} + a_{4, 2}\\
	                  &=& 0 - n + n + 1 - 2 + 1 = 0;\\
	\sum_c a'_{4, c} &=& \sum_c a_{4, c} - a_{4, 3} + a_{3, 3} - a_{4, 4} + a_{5, 4} - a_{4, 2} + a_{3, 1}\\
	                 &=& 0 - a_{4, 3} + n - a_{4, 4} - (n+1) - 1 + 2 = 0;\\
	\sum_c a'_{5, c} &=& \sum_c a_{5, c} - a_{5, 3} + a_{4, 3} - a_{5, 4} + a_{4, 4}\\
	                 &=& 0 - (n+1) + a_{4, 3} + (n+1) + a_{4, 4} = 0.
\end{array}$$
	
	Thus, in $A'$, every row and every column contains six entries that sum to zero, and every number in $[-3n, -1] \cup [1, 3n]$ is used exactly once. Note lastly that if $k \in [1, 3]$:
$$\begin{array}{lcl}
    (2j+2k-3)-(2j-1) &=& 2k - 2 \in [0, 4]\\
	(2j+2k-3)-(2j) &=& 2k - 3 \in [-1, 3]\\
	(2j+2k-2)-(2j-1) &=& 2k - 1 \in [1, 5]\\
	(2j+2k-2)-(2j) &=& 2k - 2 \in [0, 4].
\end{array}$$
	
	So in any filled cell in $A$, or equivalently in $A'$, the difference between the row and column indices is congruent to an element of $[-1, 5]$ modulo $n$. The set $[-1, 5]$ has cardinality $7$, so $A'$ uses cells only in seven consecutive diagonals. This concludes the proof, as $A'$ is the array we seek. An example of this construction is given in Figure \ref{6in10diag}.
	
	\begin{figure}[ht]
		$$\begin{array}{c || c | c | c | c | c | c | c | c | c | c |}
		P_i & P_1 & P_2 & P_3 & P_4 & P_5 & P_6 & P_7 & P_8 & P_9 & P_{10} \\ \hline \hline
		P_{i, 1} & 15 & 19 & 3 & 4 & 5 & 6 & 7 & 8 & 9 & 10 \\ \hline
		P_{i, 2} & 1 & 2 & 17 & 14 & 12 & 20 & 18 & 16 & 13 & 11 \\ \hline
		P_{i, 3} & 30 & 26 & 27 & 28 & 29 & 21 & 22 & 23 & 24 & 25 \\ \hline \hline
		\sum P_i & 46 & 47 & 47 & 46 & 46 & 47 & 47 & 47 & 46 & 46 \\ \hline
		\end{array}$$
		
		$$\begin{array}{| c | c | c | c | c | c | c | c | c | c |}
		\hline
		19 & -19 & & & & & 27 & -27 & 14 & -14 \\ \hline
		-15 & 15 & & & & & -21 & 21 & -12 & 12 \\ \hline
		\mathbf{1} & -2 & \mathbf{11} & -10 & & & & & 28 & -28 \\ \hline
		-1 & \mathbf{2} & \mathbf{10} & \mathbf{-11} & & & & & -29 & 29 \\ \hline
		26 & -26 & \mathbf{-9} & \mathbf{9} & 7 & -7 & & & & \\ \hline
		-30 & 30 & -13 & 13 & -8 & 8 & & & & \\ \hline
		& & 25 & -25 & 18 & -18 & 3 & -3 & & \\ \hline
		& & -24 & 24 & -16 & 16 & -6 & 6 & & \\ \hline
		& & & & 22 & -22 & 17 & -17 & 4 & -4 \\ \hline
		& & & & -23 & 23 & -20 & 20 & -5 & 5 \\ \hline
		\end{array}$$
		
		\caption{Top: the partition $\{P_i\}$ of $[1, 3n]$ for $n = 10$. Bottom: the corresponding $7$-diagonal
         $SMS(10;6)$\  $A'$, with the elements that differ from $A$ bolded.}
		\label{6in10diag}
	\end{figure}
\end{proof}

Now, we apply an induction argument, as in Theorem \ref{eveninodd}.

\begin{thm}
	\label{evenineven}
	Given $n \geq t > 3$ with $n, t$ even, there exists an $SMS(n;t)$. If $t < n$ with $t, n \equiv 2 \pmod 4$, this square is $(t+1)$-diagonal.
\end{thm}

\begin{proof}
	If $t$ or $n$ is a multiple of $4$, we may apply Lemma \ref{tnmult8}.
	If $t = n$, Lemma \ref{evenxeven} gives us the desired result, as the square is tight.
	Otherwise, we proceed by induction. In the base case, let $t = 6$; then there exists a $7$-diagonal
$SMS(n,t)$ by Lemma \ref{6in2mod4diag}.
	For $6 < t < n$, assume there exists a $(t-3)$-diagonal $SMS(n;t-4)$. Then there exists a $(t+1)$-diagonal $SMS(n;t)$ by Lemma \ref{add4tot}.
	Therefore, by induction and the other cases, there exist such squares for all even $t, n > 2$.
\end{proof}

\subsection{$n$ even, $t$ odd}

This is the most complex case, requiring several subcases of its own. We may proceed by an induction argument as in the previous cases, but the base case becomes much more complex. We give five such cases.

\begin{lem}\label{3ineven}
	Let $n\geq 4$ be an even integer. Then there exists an $SMS(n;3)$.
\end{lem}

\begin{proof}
	We construct two orthogonal partitions, $\{C_c\}$ and $\{R_r\}$, of the set $[-\frac{3n}{2},-1] \cup [1,\frac{3n}{2}]$. First, construct an $SMA(3,2k)$, say $A=[a_{i,j}]$,
using the construction given in the proof of Lemma \ref{3xeven}, where $2k=n$. The first partition is $\{C_c\}$, where $C_c$ is the set consisting of all the entries in the $c$th column. It is now sufficient to demonstrate the existence of a partition, $\{R_r\}$, that is orthogonal to $\{C_c\}$ and consists of $n$ sets of cardinality 3 that sum to zero. Define this partition as follows: for $1 \leq r \leq n$, let $R_r=-C_r$, where $-C_r= \{-c: c \in C_r \}$. This collection of sets clearly partitions $[-\frac{3n}{2},-1] \cup [1,\frac{3n}{2}]$ into $n$ sets of cardinality 3 that sum to 0, so it remains to show that it is orthogonal to $\{C_c\}$. Note this is equivalent to proving that the opposites of the entries of a given column $j$ in $A$ all lie in different columns. In fact, because of the zero-sum property of the rows and columns in $A$, it is sufficient to show that two of the entries in each column, when negated, lie in different columns.
	
	So fix $j$ and consider the entries $-a_{i,j}$ for $i=1,2,3$. We must consider cases depending on the value of $j$. If $j=1$, the entries are $1,\frac{3n}{2}-1$, and $-3n$. Note $-1$ lies in the second column, while $3n$ lies in the nth column.

	Next, suppose that $1<j<n$ and $j \equiv 0 \pmod 4$. Then the entries in the column are $-\left(\frac{3p_j-2}{2}\right),-3(k-p_j)$, and $\frac{3p_j-2}{2}+3(k-p_j)$, where $p_j= \lceil \frac{j}{2} \rceil$. Note $\frac{3p_{j}-2}{2}=\frac{3p_{j-1}-2}{2}$ and $j-1 \equiv 3 \pmod 4$. Therefore, $\frac{3p_{j}-2}{2}$ lies in column $j-1$. On the other hand, $3(k-p_j)=3(k-p_{j+1}+1)$, and $j+1 \equiv 1 \pmod 4$. Thus, $3(k-p_j)$ lies in column $j+1$.
	
	Now suppose that $1<j<n$ and $j \equiv 1 \pmod 4$. Then the entries in the column are $\left(\frac{3p_j-1}{2}\right),3(k-p_j+1)$, and $-\left(\frac{3p_j-1}{2}\right)-3(k-p_j+1)$.  Note $-\left(\frac{3p_{j}-1}{2}\right)=-\left(\frac{3p_{j+1}-1}{2}\right)$ and $j+1 \equiv 2 \pmod 4$. Therefore, $-\left(\frac{3p_{j}-1}{2}\right)$ lies in column $j+1$. On the other hand, $-3(k-p_j+1)=-3(k-p_{j-1})$, and $j-1 \equiv 0 \pmod 4$. Thus, $3(k-p_j)$ lies in column $j-1$.
	
	Next, suppose that $1<j<n$ and $j \equiv 2 \pmod 4$. Then the entries in the column are $-\left(\frac{3p_j-1}{2}\right),-3(k-p_j)$, and $\frac{3p_j-1}{2}-3(k-p_j)$. Note $\frac{3p_{j}-1}{2}=\frac{3p_{j-1}-1}{2}$ and $j-1 \equiv 1 \pmod 4$. Therefore, $\frac{3p_{j}-1}{2}$ lies in column $j-1$. On the other hand, $3(k-p_j)=3(k-p_{j+1}+1)$, and $j+1 \equiv 3 \pmod 4$. Thus, $3(k-p_j)$ lies in column $j+1$.
	
	Finally, suppose $1<j<n$ and $j \equiv 3 \pmod 4$. Then the entries in the column are $\left(\frac{3p_j-2}{2}\right),3(k-p_j+1)$, and $-\left(\frac{3p_j-2}{2}\right)-3(k-p_j+1)$.  Note $-\left(\frac{3p_{j}-2}{2}\right)=-\left(\frac{3p_{j+1}-2}{2}\right)$ and $j+1 \equiv 0 \pmod 4$. Therefore, $-\left(\frac{3p_{j}-2}{2}\right)$ lies in column $j+1$. On the other hand, $-3(k-p_j+1)=-3(k-p_{j-1})$, and $j-1 \equiv 2 \pmod 4$. Thus, $3(k-p_j)$ lies in column $j-1$.
	
	It is unnecessary to check the last column because if the last column did have the property that two of its elements, when negated, were in the same column, then another column would have that same property.
	
	Let $B$ be the $n \times n$ array where cell $(i,j)$ contains the element common to $R_i$ and $C_j$ if such an element exists, and is left blank otherwise. Then $B$ contains 3 entries in each row and column that sum to zero, as required (see Figure \ref{3in6}).
\end{proof}
	
	\begin{figure}[ht]
		$$\begin{array}{|c|c|c|c|c|c|}
		\hline
		& -1 & -8  & & & 9  \\
		\hline
		1 &  & 6 & & -7 & \\
		\hline
		8 & -6 &  & -2 & &  \\
		\hline
		&  & 2 &  & 3 & -5  \\
		\hline
		& 7 & & -3 &  & -4 \\
		\hline
		-9& & & 5 & 4 &  \\
		\hline
		\end{array}$$
		
		\caption{An $SMS(6;3)$, using the method of Lemma \ref{3ineven}.}
		\label{3in6}
	\end{figure}

\begin{lem}
	\label{3in0mod4diag}
	Let $n=4k$ with $k\geq 1$. Then there exists a diagonal $SMS(n;3)$.
\end{lem}

\begin{proof}
	We  define  three finite sequences $a_i, b_i, c_i$, with $i \in [1, n]$, which together contain every integer in $[-\frac{3n}{2},-1]\cup[1,\frac{3n}{2}]$ exactly once.
	
	$$a_i = \begin{cases}
	-2 - 3k - 3\frac{i-1}{2} & \text{if } i < 2k \text{ and } 2 \nmid i \\
	-2 + 9k - 3\frac{i-1}{2} & \text{if } i > 2k \text{ and } 2 \nmid i \\
	-2 + 3k - 3\frac{i}{2} & \text{if } i < 4k \text{ and } 2 \mid i \\
	-2 + 3k & \text{if } i = 4k
	\end{cases}$$
	
	$$b_i = \begin{cases}
	3i & \text{if } i \leq 2k \\
	-12k + 3i & \text{if } 2k < i < 4k \\
	-6k & \text{if } i = 4k
	\end{cases}$$
	
	$$c_i = a_i + 1$$
	
	First, we must prove that these sequences, together, contain every number in the specified range. Let $x$ be an integer with $|x| \in [1, 6k]=[1,\frac{3n}{2}]$.
	
	If $x = 3y$ for integer $y > 0$, we have $b_y = x$.
	
	If $x = -3y$ for integer $y > 0$, either $y = 2k$ or $y < 2k$. In the former case, $x = b_{4k}$. In the latter case, $x = b_i$ where $i = 4k - y$.
	
	If $x = 3y - 2$ for $-2k < y \leq -k$, then $x = a_i$ where $i = 1 - 2k - 2y$.
	
	If $x = 3y - 2$ for $-k < y < k$, then $x = a_i$ where $i = 2k - 2y$.
	
	If $x = 3k - 2$, then $x = a_{4k}$.
	
	If $x = 3y - 2$ for $k < y \leq 2k$, then $x = a_i$ where $i = 1 + 6k - 2y$.
	
	If $x = 3y - 1$ for $-2k < y \leq 2k$, then $x = c_i$ such that $x - 1 = a_i$, which is given by the previous cases.
	
	Note that $\{a_i\}$, $\{b_i\}$, and $\{c_i\}$ are all disjoint, because $a_i \equiv 1 \pmod 3$, $b_i \equiv 0 \pmod 3$, and $c_i \equiv 2 \pmod 3$ for all $i$.
	As $\{a_i\} \cup \{b_i\} \cup \{c_i\} \supseteq [-6k, -1] \cup [1, 6k]$, and these sets both have cardinality $12k$, it follows that they are the same set, so $a_i, b_i, c_i$ contain between them every integer in $[-6k, -1] \cup [1, 6k]$ exactly once.
	
	Now we will prove the following additional property of these sequences: letting $a_{n+1} = a_1$, $b_{n+1} = b_1$, and $c_{n+1} = c_1$, then for all $i \in [1, n]$, we have $a_i + b_{i+1} + c_{i+1} = a_{i+1} + b_{i+1} + c_i = 0$. Note that we need only prove that the first expression is zero, as $a_{i+1} + c_i = c_{i+1} - 1 + a_i + 1 = a_i + c_{i+1}$.

	Let $i < 2k$ with $2 \nmid i$. Then
$$	
	a_i + b_{i+1} + c_{i+1}
	= -2 - 3k - 3\frac{i-1}{2} + 3(i+1) + -2 + 3k - 3\frac{i+1}{2} + 1=0.\\
$$
	
	Let $i < 2k$ with $2 \mid i$. Then
$$	
	a_i + b_{i+1} + c_{i+1}
	= -2 + 3k - 3\frac{i}{2} + 3(i+1) + -2 - 3k - 3\frac{i}{2} + 1=0.\\
$$
	
	Let $2k < i < 4k - 1$ with $2 \nmid i$. Then
$$	
	a_i + b_{i+1} + c_{i+1}
	= -2 + 9k - 3\frac{i-1}{2} - 12k + 3(i+1) + -2 + 3k - 3\frac{i+1}{2} + 1=0.\\
$$
	
	Let $2k \leq i < 4k$ with $2 \mid i$. Then
$$	
	a_i + b_{i+1} + c_{i+1}
	= -2 + 3k - 3\frac{i}{2} + -12k + 3(i+1) + -2 + 9k - 3\frac{i}{2} + 1=0.\\
$$
	
	Let $i = 4k - 1$. Then
$$	
	a_i + b_{i+1} + c_{i+1}
	= -2 + 9k - 3\frac{4k - 2}{2} + -6k + -2 + 3k + 1=0.\\
$$
	
	Let $i = 4k$. Then
$$
	a_i + b_{i+1} + c_{i+1}
	= -2 + 3k + 3 + -2 - 3k - 3\frac{1 - 1}{2} + 1=0.
$$
	
	Now we  define  an $n \times n$ array $A$ as follows:
	
	$$A_{i, j} = \begin{cases}
	c_i & \text{if } j = i \\
	b_{i+1} & \text{if } j \equiv i + 1 \pmod n \\
	a_{i+1} & \text{if } j \equiv i + 2 \pmod n
	\end{cases}$$
	
	with the remaining cells empty (see Figure \ref{3in8diag}). Then the $i$th row of $A$ contains $c_i$, $b_{i+1}$ and $a_{i+1}$, which add to 0. The $j$th column contains $c_j$, $b_j$, and $a_{j-1}$ (letting $a_0 = a_n$), which also add to 0. The array $A$ thus contains three consecutive diagonals filled with the integers in $[-\frac{3n}{2}, -1]$ and $[1, \frac{3n}{2}]$ such that the sum of the three integers in each row and in each column is zero.
\end{proof}
	
	\begin{figure}[ht]
		$$\begin{array}{|c|c|c|c|c|c|c|c|}
		\hline
		3 & -8 & & & & & & 5 \\ \hline
		-7 & 6 & 1 & & & & & \\ \hline
		& 2 & 9 & -11 & & & & \\ \hline
		& & -10 & 12 & -2 & & & \\ \hline
		& & & -1 & -9 & 10 & & \\ \hline
		& & & & 11 & -6 & -5 & \\ \hline
		& & & & & -4 & -3 & 7 \\ \hline
		4 & & & & & & 8 & -12 \\ \hline
		\end{array}$$
		
		\caption{A diagonal $SMS(8;3)$, using the method of Lemma \ref{3in0mod4diag}.}
		\label{3in8diag}
	\end{figure}

\begin{lem}
	\label{5in0mod4diag}
	Let $n = 4k$ with $k > 1$. Then there exists a diagonal $SMS(n;5)$.
\end{lem}

\begin{proof}
	We  define  five finite sequences $a_i, b_i, c_i, d_i, e_i$, with $i \in [1, n]$, which together contain every integer in $[-\frac{5n}{2}, -1]\cup [1, \frac{5n}{2}]$ exactly once.
	$$a_i = \begin{cases}
	-10j - 18 & \text{if } i = 4j+1 \text{ and } j < k - 1 \\
	-8 & \text{if } i = 4k-3 \\
	-10j - 13 & \text{if } i = 4j+2 \text{ and } j < k - 1 \\
	-3 & \text{if } i = 4k-2 \\
	10k - 10j - 3 & \text{if } i = 4j+3 \\
	10k - 10j - 18 & \text{if } i = 4j+4 \text{ and } j < k - 1 \\
	10k - 8 & \text{if } i = 4k
	\end{cases}$$
	$$b_i = \begin{cases}
	-5k - 5j - 4 & \text{if } i = 2j+1 \text{ and } j < k \\
	15k - 5j - 4 & \text{if } i = 2j+1 \text{ and } k \leq j < 2k-2 \\
	-5k + 6 & \text{if } i = 4k-3 \\
	-5k + 1 & \text{if } i = 4k-1 \\
	-5k + 5j + 11 & \text{if } i = 2j+2,
	\end{cases}$$
	$$c_i = \begin{cases}
	5i & \text{if } i \leq 2k \\
	-20k + 5i & \text{if } 2k < i < 4k \\
	-10k & \text{if } i = 4k,
	\end{cases}$$
	$d_i = b_i + 3$ and
    $e_i = a_i + 1.$
	
	First, we must prove that these sequences, together, contain every number in the specified range. Let $x$ be an integer with $|x| \in [1, 10k]=[1, \frac{5n}{2}]$.
	
	If $x = 5y$ for integer $y > 0$, we have $c_y = x$.
	
	If $x = -5y$ for integer $y > 0$, either $y = 2k$ or $y < 2k$. In the former case, $x = c_{4k}$. In the latter case, $x = c_i$ where $i = 4k - y$.
	
	If $x = 5y - 3$ for odd $y$ such that $-2k < y < -1$, then $x = a_i$ where $i = -2y - 5$.
	
	If $x = 5y - 3$ for $y = -1$, then $x = a_{4k-3}$.
	
	If $x = 5y - 3$ for odd $y$ such that $-1 < y < 2k-1$, then $x = a_i$ where $i = 4k - 2y - 2$.
	
	If $x = 5y - 3$ for $y = 2k - 1$, then $x = a_{4k}$.
	
	If $x = 5y - 3$ for even negative $y$, then $x = a_i$ where $i = -2y - 2$.
	
	If $x = 5y - 3$ for $y = 0$, then $x = a_{4k-2}$.
	
	If $x = 5y - 3$ for even positive $y$, then $x = a_i$ where $i = 4k - 2y + 3$.
	
	If $x = 5y - 2$, then there is an $i$ such that $x - 1 = a_i$, per the above. Then $x = e_i$.
	
	If $x = 5y - 4$ for $-2k < y \leq -k$, then $x = b_i$ where $i = -2k - 2y + 1$.
	
	If $x = 5y - 4$ for $y = -k + 1$, then $x = b_{4k-1}$.
	
	If $x = 5y - 4$ for $y = -k + 2$, then $x = b_{4k-3}$.
	
	If $x = 5y - 4$ for $-k + 3 \leq y \leq k + 2$, then $x = b_i$ where $i = 2k + 2y - 4$.
	
	If $x = 5y - 4$ for $k + 3 \leq y \leq 2k$, then $x = b_i$ where $i = 6k - 2y + 1$.
	
	If $x = 5y - 1$, then there is an $i$ such that $x - 3 = b_i$, per the above. Then $x = d_i$.
	
	Note that $\{a_i\}$, $\{b_i\}$, $\{c_i\}$, $\{d_i\}$, and $\{e_i\}$ are all disjoint, because $a_i \equiv 2 \pmod 5$, $b_i \equiv 1 \pmod 5$, $c_i \equiv 0 \pmod 5$, $d_i \equiv -1 \pmod 5$, and $e_i \equiv -2 \pmod 5$ for all $i$.
	
	As $\{a_i\} \cup \{b_i\} \cup \{c_i\} \cup \{d_i\} \cup \{e_i\} \supseteq [-10k, -1] \cup [1, 10k]$, and these sets both have cardinality $20k$, it follows that they are the same set, so $a_i, b_i, c_i, d_i, e_i$ contain between them every integer in $[-10k, -1] \cup [1, 10k]$ exactly once.
	
	For convenience, the subscripts of these sequences will be treated as elements of $\mathbb{Z}_n$. For example, the notation $a_{n+7}$ will refer to $a_7$.
	
	Now consider the expression $S_i = a_{i-2} + b_{i-1} + c_i + d_i + e_i$. We will compute the value of this expression for all $i \in [1, n]$.
	
	If $i = 1$, then $S_i = a_{i-2} + b_{i-1} + c_i + d_i + e_i = 7 + 5k + 6 + 5 - 5k - 4 + 3 - 18 + 1 = 0$.
	
	If $i = 2$, then $S_i = a_{i-2} + b_{i-1} + c_i + d_i + e_i = 10k - 8 - 5k - 4 + 10 - 5k + 11 + 3 - 13 + 1 = 0$.
	
	If $i = 4j + 3$ with $j \geq 0$ and $i \leq 2k$, then $S_i = a_{i-2} + b_{i-1} + c_i + d_i + e_i = -10j - 18 - 5k + 10j + 11 + 20j + 15 - 5k - 5(2j + 1) - 4 + 3 + 10k - 10j - 3 + 1 = 0$.
	
	If $i = 4j + 4$ with $j \geq 0$ and $i \leq 2k$, then $S_i = a_{i-2} + b_{i-1} + c_i + d_i + e_i = -10j - 13 - 5k - 5(2j + 1) - 4 + 20j + 20 - 5k + 5(2j + 1) + 11 + 3 + 10k - 10j - 18 + 1 = 0$.
	
	If $i = 4j + 1$ with $j > 0$ and $i \leq 2k$, then $S_i = a_{i-2} + b_{i-1} + c_i + d_i + e_i = 10k - 10(j - 1) - 3 - 5k + 5(2j - 1) + 11 + 20j + 5 - 5k - 10j - 4 + 3 - 10j - 18 + 1 = 0$.
	
	If $i = 4j + 2$ with $j > 0$ and $i \leq 2k$, then $S_i = a_{i-2} + b_{i-1} + c_i + d_i + e_i = 10k - 10(j - 1) - 18 - 5k - 10j - 4 + 20j + 10 - 5k + 10j + 11 + 3 - 10j - 13 + 1 = 0$.
	
	If $i = 4j + 1 > 2k$ with $i < 4k - 3$, then $S_i = a_{i-2} + b_{i-1} + c_i + d_i + e_i = 10k - 10(j - 1) - 3 - 5k + 5(2j - 1) + 11 - 20k + 20j + 5 + 15k - 10j - 4 + 3 - 10j - 18 + 1 = 0$.
	
	If $i = 4j + 2 > 2k$ with $i < 4k - 3$, then $S_i = a_{i-2} + b_{i-1} + c_i + d_i + e_i = 10k - 10(j - 1) - 18 + 15k - 10j - 4 - 20k + 20j + 10 - 5k + 10j + 11 + 3 - 10j - 13 + 1 = 0$.
	
	If $i = 4j + 3 > 2k$ with $i < 4k - 3$, then $S_i = a_{i-2} + b_{i-1} + c_i + d_i + e_i = -10j - 18 - 5k + 10j + 11 - 20k + 20j + 15 + 15k - 5(2j + 1) - 4 + 3 + 10k - 10j - 3 + 1 = 0$.
	
	If $i = 4j + 4 > 2k$ with $i < 4k - 3$, then $S_i = a_{i-2} + b_{i-1} + c_i + d_i + e_i = -10j - 13 + 15k - 5(2j + 1) - 4 - 20k + 20j + 20 - 5k + 5(2j + 1) + 11 + 3 + 10k - 10j - 18 + 1 = 0$.
	
	If $i = 4k - 3$, then $S_i = a_{i-2} + b_{i-1} + c_i + d_i + e_i = 10k - 10(k - 2) - 3 - 5k + 5(2k - 3) + 11 - 20k + 20k - 15 - 5k + 6 + 3 - 8 + 1 = 0$.
	
	If $i = 4k - 2$, then $S_i = a_{i-2} + b_{i-1} + c_i + d_i + e_i = 10k - 10(k - 2) - 18 - 5k + 6 - 20k + 20k - 10 - 5k + 5(2k - 2) + 11 + 3 - 3 + 1 = 0$.
	
	If $i = 4k - 1$, then $S_i = a_{i-2} + b_{i-1} + c_i + d_i + e_i = -8 - 5k + 5(2k - 2) + 11 - 20k + 20k - 5 - 5k + 1 + 3 + 10k - 10(k - 1) - 3 + 1 = 0$.
	
	If $i = 4k$, then $S_i = a_{i-2} + b_{i-1} + c_i + d_i + e_i = -3 - 5k + 1 - 10k - 5k + 5(2k - 1) + 11 + 3 + 10k - 8 + 1 = 0$.
	
	So $S_i = 0$ for all $i \in [1, n]$.
	
	Now consider $S'_i = a_i + b_i + c_i + d_{i-1} + e_{i-2}$. From the definitions of $d_i$ and $e_i$, we see that $S'_i = e_i - 1 + d_i - 3 + c_i + b_{i-1} + 3 + a_{i-2} + 1 = e_i + d_i + c_i + b_{i-1} + a_{i-2} = S_i = 0$.
	
	We will now define an $n \times n$ square array $A=[a_{i,j}]$ in which we fill five consecutive diagonals. Again, the indices of $A$ will be considered as elements of $\mathbb{Z}_n$. For $i \in [1, n]$, we let
	$a_{i, i+2} = a_i$,
	$a_{i, i+1} = b_i$,
	$a_{i, i} = c_i$,
	$a_{i+1, i} = d_i$ and
	$a_{i+2, i} = e_i$,
	with the other cells empty (see Figure \ref{5in8diag}).
	
	Clearly this fills precisely five adjacent diagonals of $A$ with the elements in $[-10k, -1] \cup [1, 10k]$.
	Now in row $i$ of $A$, the five cells filled are $a_{i, i+2} = a_i$, $a_{i, i+1} = b_i$, $a_{i, i} = c_i$, $a_{i, i-1} = d_{i-1}$, and $a_{i, i-2} = e_{i-2}$. The sum of these five cells is $S'_i = 0$.
	In column $i$ of $A$, the five cells filled are $a_{i+2, i} = e_i$, $a_{i+1, i} = d_i$, $a_{i, i} = c_i$, $a_{i-1, i} = b_{i-1}$, and $a_{i-2, i} = a_{i-2}$. The sum of these five cells is $S_i = 0$.
	Therefore $A$ is a diagonal $SMS(n;5),$ where $n\equiv 0 \pmod 4$ and $n\geq 8$.
\end{proof}	

	\begin{figure}[b]
		$$\begin{array}{|c|c|c|c|c|c|c|c|}
		\hline
		5 & 6 & -13 & & & & -17 & 19 \\ \hline
		9 & 10 & -19 & -8 & & & & 8 \\ \hline
		-12 & -16 & 15 & 1 & 12 & & & \\ \hline
		& -7 & 4 & 20 & -14 & -3 & & \\ \hline
		& & 13 & -11 & -15 & -4 & 17 & \\ \hline
		& & & -2 & -1 & -10 & 11 & 2 \\ \hline
		-18 & & & & 18 & 14 & -5 & -9 \\ \hline
		16 & 7 & & & & 3 & -6 & -20 \\ \hline
		\end{array}$$
		
		\caption{A diagonal $SMS(8;5)$, using the method of Lemma \ref{5in0mod4diag}.}
		\label{5in8diag}
	\end{figure}

\begin{lem}
	\label{5in2mod4diag}
Let $n = 4k+2$ with $k \geq 1$. Then there exists a diagonal $SMS(n;5)$.
\end{lem}

\begin{proof} Let $n=4k+2$. Define $A=[a_{i,j}]$ as follows.
	If $i=j$
	$$a_{i,j}= \begin{cases}
	5i & 1 \leq i \leq \frac{n}{2} \\
	-5(n-i) & \frac{n}{2}<i<n \\
	-\frac{5n}{2} & i=n.\\
	\end{cases}$$
	
	If $i \equiv j-2 \pmod n$
	$$a_{i,j}= \begin{cases}
	-3-5(i-1) & i \leq \frac{n}{2} \\
	5(n-i)+2 & i> \frac{n}{2}.
	\end{cases}$$
	
	If $i \equiv j-1 \pmod n$
	$$a_{i,j}= \begin{cases}
	-9-5k+5 \left(\frac{i-1}{2} \right) & i \text{ odd} \\
	5k+1+5\left(\frac{i-2}{2}\right) & i \text{ even and } i < 2k+4\\
	5k+1+5\left(\frac{i-2}{2}\right)-5n & i \text{ even and } 2k+4 \leq i<n\\
	\frac{5n-26}{4} & i=n.
	\end{cases}$$
	
	If $i \equiv j+1 \pmod n,$ then $a_{i,j}= a_{i-1,j+1}+3$;
    if $i \equiv j+2 \pmod n,$ then
	$a_{i,j}= a_{i-2,j+2}+1,$
	where for convenience we define $a_{i+n,j}=a_{i,j}$ and $a_{i,j+n}=a_{i,j}$. Let all other cells be blank.

	Now, for each diagonal, or congruence class of $i-j \pmod n$, this function is defined on $n$ positive integer values for $i$. Hence, the function as a whole has a range $R$ with $|R| \leq 5n$. Consequently, it suffices to show that for every $x \in [-\frac{5n}{2},-1] \cup [1,\frac{5n}{2}]$, there exists an $a_{i,j}$ with $a_{i,j}=x$ as according to the function defined above.

	Let $x \in [-\frac{5n}{2},-1] \cup [1,\frac{5n}{2}]$ be arbitrary. If $x \equiv 0 \pmod 5$, then either $x=5\ell$ or $x=-5 \ell$ for some $\ell \in [1,\frac{n}{2}]$. If $x>0,$ let $i=j=\ell$. If $-\frac{5n}{2}<x<0$, then let $i=j=n-\ell$. Finally, if $x=-\frac{5n}{2}$, let $i=j=n$.
	
	If $x \equiv 1 \pmod 5$, we consider some cases. If $-9-5k\leq x<0$, let $i=\frac{2(x+9+5k)+5}{5}$. If $x < -9k-5$, let $i= \frac{2(x+5n-5k)+8}{5}$. If $0<x<5k+1$ and $x \neq \frac{5n-26}{4}$, then let $i=\frac{2(x+9+5k)+5}{5}$. If $x \geq 5k+1$, then let $i=\frac{2(x-5k-1)+8}{5}$. If $x= \frac{5n-26}{4}$, let $i=n$. For all these cases, let $j \equiv i+1 \pmod n$.
	
	If $x \equiv 2 \pmod 5$, we consider some cases. If $x>0$, then let $i=n-\frac{x-2}{5}$. If $x<0$, let $i=-\left(\frac{x+3}{5}\right)+1$.
	
	It is routine to verify that all these choices of $x$ yield integer values of $i$ between $1$ and $n$, and that for the corresponding $i$ value $a_{i,j}=x$ according to the function defined above.
	
	To cover the remaining congruence classes that $x$ could assume, note that all of the terms $a_{i,j}$ satisfying $i \equiv j-1 \pmod n$ are translates of elements of the array when $i \equiv j-1$. In particular, $3$ is added to each of the entries. Hence, the fact that the function covers all values of $x \equiv 4 \pmod n$ follows from the fact that it covers all values with $x \equiv 1 \pmod n$. An analogous argument shows that the array contains all values of $x$ when $x \equiv 3 \pmod n$. We conclude that the array contains every element in $[-\frac{5n}{2},-1] \cup [1,\frac{5n}{2}]$ exactly once.
	
	Now we must show that all rows and columns in the array sum to zero. First consider the rows. We again have several cases to consider.
	
	If $i=1,$ then $\sum_{j} a_{i,j}=5(1)-3-5(1-1)-9-5k+5\left(\frac{1-1}{2}\right)+\frac{5n-26}{4}+3+5(n-(n-1))+2+1=5-3-9-4-5\left(\frac{n-2}{4}\right)+\frac{5n-26}{4}+3+5+3=5-3-9-4+3+5+3=0$.
	
	If $i=2,$ $\sum_{j} a_{i,j}= 5(2)-3-5(2-1)+5k+1+5 \left(\frac{2-2}{2}\right)-9-5k+5\left(\frac{2-2}{2}\right)+3+5(n-n)+2+1=10-3-5+1-9+3+2+1=0$.
	
	If $2<i \leq \frac{n}{2}$ and $i$ is odd, then $\sum_{j} a_{i,j}=5i-3-5(i-1)-9-5k+5 \left(\frac{i-1}{2}\right)+5k+1+5\left(\frac{i-3}{2}\right)+3-3-5(i-3)+1=-3+5-9-10+1+3-3+15+1=0$.
	
	If $2<i \leq \frac{n}{2}$ and $i$ is even, then $\sum_{j} a_{i,j}= 5i-3-5(i-1)+5k+1+5 \left(\frac{i-2}{2}\right)-9-5k+5\left(\frac{i-2}{2}\right)+3-3-5(i-3)+1=-3+5+1-10-9+3-3+15+1=0$.
	
	If $\frac{n}{2}<i<2k+5$ and $i$ is odd, then $\sum_{j} a_{i,j}=-5(n-i)+5(n-i)+2-9-5k+5\left(\frac{i-1}{2}\right)+5k+1+5\left(\frac{i-3}{2}\right)+3-3-5(i-3)+1=2-9-10+1+3-3+15+1=0$.
	
	If $\frac{n}{2}<i<2k+4$ and $i$ is even, then $\sum_{j} a_{i,j}= -5(n-i)+5(n-i)+2+5k+1+5 \left(\frac{i-2}{2}\right)-9-5k+5\left(\frac{i-2}{2}\right)+3+3-3-5(i-3)=2+1-10-9+3-3+15+1=0$.
	
	If $2k+5 \leq i< n$ and $i$ is odd, then $\sum_{j}a_{i,j}=-5(n-i)+5(n-i)+2-9-5k+5\left(\frac{i-1}{2}\right)+5k+1+5\frac{i-3}{2}-5n+3+5(n-(i-2))+2+1=2-9-10+1+3+10+2+1=0$.
	
	If $2k+4 \leq i< n$ and $i$ is even, then $\sum_{j} a_{i,j}= -5(n-i)+5(n-i)+2+5k+1+5 \left(\frac{i-2}{2}\right)-5n-9-5k+5\left(\frac{i-2}{2}\right)+3+5(n-(i-2))+2+1=2+1-10-9+10+1+3+2=0$.
	
	If $i=n$, then the entries in the row must sum to zero, because all the other rows sum to zero and the sum of all entries in the array is 0.
	
	Now we consider similar cases with the columns.
	
	If $j=1$, then $\sum_{i} a_{i,j}=5(1)+5(n-(n-1))+2+\frac{5n-26}{4}-9-5k+5\left(\frac{1-1}{2}\right)+3-3-5(1-1)+1=5+5+\frac{5n-26}{4}-\frac{5(n-2)}{4}-9+3-3+1=10-4-9+2+3-3+1=0$.
	
	If $j=2$, then $\sum_{i} a_{i,j}=5(2)+5(n-n)+2-9-5k+5\left(\frac{1-1}{2}\right)+5k+1+\left(\frac{2-2}{2}\right)+3-3-5(1)+1=10+2-9+3-3+1-5+1=0$.
	
	If $2<j \leq \frac{n}{2}$ and $j$ is odd, then $\sum_{i} a_{i,j}=5j-3-5(j-3)+5k+1+5\left(\frac{j-3}{2}\right)-9-5k+5\left(\frac{j-1}{2}\right)+3-3-5(j-1)+1=-3+15+1-10-9+3-3+5+1=0$.
	
	If $2<j \leq \frac{n}{2}$ and $j$ is even, then $\sum_{i} a_{i,j}=5j-3-5(j-3)-9-5k+5\left(\frac{j-2}{2}\right)+5k+1+5\left(\frac{j-2}{2}\right)+3-3-5(j-1)+1=-3+15-9+1-10+3-3+5+1=0$.
	
	If $\frac{n}{2}<j <2k+5$ and $j$ is odd, then $\sum_{i} a_{i,j}=-5(n-j)-3-5(j-3)+5k+1+5\left(\frac{j-3}{2}\right)-9-5k+5\left(\frac{j-1}{2}\right)+3+5(n-j)+2+1=-3+15+1-10-9+3+2+1=0$.
	
	If $\frac{n}{2}<j <2k+4$ and $j$ is even, then $\sum_{i} a_{i,j}=-5(n-j)-3-5(j-3)-9-5k+5\left(\frac{j-2}{2}\right)+5k+1+5\left(\frac{j-2}{2}\right)+3+5(n-j)+2+1=-3+15-9-10+1+3+2+1=0$.
	
	If $2k+5 \leq j < n$ and $j$ is odd, then $\sum_{i} a_{i,j}=-5(n-j)+5(n-(j-2))+2+5k+1+5 \left(\frac{j-3}{2}\right)-5n-9-5k+5\left(\frac{j-1}{2}\right)+3+5(n-j)+2+1=10+2+1-10-9+3+2+1=0$.
	
	If $2k+4 \leq j < n$ and $j$ is even, then $\sum_{i} a_{i,j}= -5(n-j)+5(n-(j-2))+2-9-5k+5\left(\frac{j-2}{2}\right)+5k+1+5\left(\frac{j-2}{2}\right)-5n+3+5(n-j)+2+1=10+2-9-10+1+3+2+1=0$.

	If $i=n$, then the entries in the column must sum to zero, because all the other columns sum to zero and the sum of all entries in the array is 0 (see Figure \ref{5in10diag}).
	
	We conclude that if $n \equiv 2 \pmod 4$, there exists a diagonal $SMS(n;5)$
	
	\begin{figure}[b]
		$$\begin{array}{|c|c|c|c|c|c|c|c|c|c|}
		\hline
		5 & -19 & -3 & & & & & & 8 & 9\\
		\hline
		-16 & 10 & 11 & -8 & & & & & & 3 \\
		\hline
		-2 & 14 & 15 & -14 & -13 & & & & & \\
		\hline
		& -7 & -11 & 20 & 16 & -18 & & & &  \\
		\hline
		& &-12 & 19 & 25 & -9 & -23 & & & \\
		\hline
		& & & -17 & -6 & -20 & 21 & 22 & &\\
		\hline
		& & & & -22 & 24 & -15 & -4 & 17 & \\
		\hline
		&  & & & & 23 & -1 & -10 & -24 & 12\\
		\hline
		7 & & & & & & 18 &-21 & -5 & 1\\
		\hline
		6 & 2 & & & & & & 13 & 4 & -25\\
		\hline
		\end{array}$$
		
		\caption{A diagonal $SMS(10;5)$, using the method of Lemma \ref{5in2mod4diag}.}
		\label{5in10diag}
	\end{figure}
\end{proof}

\begin{lem}
	\label{7in2mod4diag}
	Let $n = 4k+2$ with $k \geq 2$. Then there exists a diagonal $SMS(n;7)$.
\end{lem}

\begin{proof}
	We  define  seven finite sequences $a_i, b_i, c_i, d_i, e_i, f_i, g_i$, with $i \in [1, n]$, which together contain every integer in the required set $X = [-\frac{7n}{2}, -1] \cup [1, \frac{7n}{2}]$ exactly once.
	
	$$a_i = \begin{cases}
	-7i + 3 & \text{if } i \leq 2k+1 \\
	28k - 7i + 17 & \text{if } i > 2k+1,
	\end{cases}$$
	
	$$b_i = \begin{cases}
	7i - 12 & \text{if } i \leq 2k+2 \\
	-28k + 7i - 26 & \text{if } i > 2k+2,
	\end{cases}$$
	
	$$c_i = \begin{cases}
	7k - 7j + 1 & \text{if } i = 2j \text{ and } i < 4k+2 \\
	7k + 1 & \text{if } i = 4k+2 \\
	-7k - 7j - 6 & \text{if } i = 2j+1 \text{ and } i \leq 2k+1 \\
	21k - 7j + 8 & \text{if } i = 2j+1 \text{ and } i > 2k+1,
	\end{cases}$$
	
	$$d_i = \begin{cases}
	7i & \text{if } i \leq 2k+1 \\
	-28k + 7i - 14 & \text{if } 2k+1 < i < 4k+2 \\
	-14k - 7 & \text{if } i = 4k+2,
	\end{cases}$$
	$e_i = c_i + 5$,
	$f_i = b_i + 3$, and
	$g_i = a_i + 1.$
	
	First, we must prove that these sequences, together, contain every number in the specified range. Let $x \in X$.
	
	If $x = 7y$ for integer $y > 0$, we have $c_y = x$.
	
	If $x = -7y$ for integer $y > 0$, either $y = 2k+1$ or $y < 2k+1$. In the former case, $x = d_{4k+2}$. In the latter case, $x = d_i$ where $i = 4k + 2 - y$.
	
	If $x = 7y+1$ for integer $y$ such that $y < -k$, then $x = c_i$ where $i = -2k - 2y - 1$.
	
	If $x = 7y+1$ for integer $y$ such that $-k \leq y < k$, then $x = c_i$ where $i = 2k - 2y + 1$.
	
	If $x = 7k+1$, then $x = c_{4k+2}$.
	
	If $x = 7y+1$ for integer $y$ such that $k < y$, then $x = c_i$ where $i = 6k - 2y + 1$.
	
	If $x = 7y+2$ for integer $y$ such that $y < -1$, then $x = b_i$ where $i = 4k + y + 4$.
	
	If $x = 7y+2$ for integer $y$ such that $y \geq -1$, then $x = b_i$ where $i = y + 2$.
	
	If $x = 7y+3$ for integer $y$ such that $y < 0$, then $x = a_i$ where $i = -y$.
	
	If $x = 7y+3$ for integer $y$ such that $y \geq 0$, then $x = a_i$ where $i = 4k - y + 2$.
	
	If $x = 7y+4$ for integer $y$, then there exists an $i$ such that $x - 1 = a_i$, and $x = g_i$.
	
	If $x = 7y+5$ for integer $y$, then there exists an $i$ such that $x - 3 = b_i$, and $x = f_i$.
	
	If $x = 7y+6$ for integer $y$, then there exists an $i$ such that $x - 5 = c_i$, and $x = e_i$.
	
	Note that $\{a_i\}$, $\{b_i\}$, $\{c_i\}$, $\{d_i\}$, $\{e_i\}$, $\{f_i\}$, and $\{g_i\}$ are all disjoint, because $a_i \equiv 3 \pmod 7$, $b_i \equiv 2 \pmod 7$, $c_i \equiv 1 \pmod 7$, $d_i \equiv 0 \pmod 7$, $e_i \equiv -1 \pmod 7$, $f_i \equiv -2 \pmod 7$, and $g_i \equiv -3 \pmod 7$ for all $i$.
	As $\{a_i\} \cup \{b_i\} \cup \{c_i\} \cup \{d_i\} \cup \{e_i\} \cup \{f_i\} \cup \{g_i\} \supseteq X$, and these sets both have cardinality $28k + 14$, it follows that they are the same set, so $a_i, b_i, c_i, d_i, e_i, f, g$ contain between them every integer in $X$ exactly once.
	
	For convenience, the subscripts of these sequences will be treated as elements of $\mathbb{Z}_n$. For example, the notation $a_{n+7}$ will refer to $a_7$.
	Now consider the expression $S_i = a_{i-3} + b_{i-2} + c_{i-1} + d_i + e_i + f_i + g_i$. We will compute the value of this expression for all $i \in [1, n]$.
	
	If $i = 1$, then $S_i = a_{i-3} + b_{i-2} + c_{i-1} + d_i + e_i + f_i + g_i = 28k - 7(4k) + 17 - 28k + 7(4k+1) - 26 + 7k + 1 + 7 - 7k - 6 + 5 + 7 - 12 + 3 - 7 + 3 + 1 = 0$.
	
	If $i = 2$, then $S_i = a_{i-3} + b_{i-2} + c_{i-1} + d_i + e_i + f_i + g_i = 28k - 7(4k+1) + 17 - 28k + 7(4k+2) - 26 - 7k - 6 + 14 + 7k - 7 + 1 + 5 + 14 - 12 + 3 - 14 + 3 + 1 = 0$.
	
	If $i = 3$, then $S_i = a_{i-3} + b_{i-2} + c_{i-1} + d_i + e_i + f_i + g_i = 28k - 7(4k+2) + 17 + 7 - 12 + 7k - 7 + 1 + 21 - 7k - 7 - 6 + 5 + 21 - 12 + 3 - 21 + 3 + 1 = 0$.
	
	If $i = 2j \leq 2k+1$, then $S_i = a_{i-3} + b_{i-2} + c_{i-1} + d_i + e_i + f_i + g_i = -7(2j-3) + 3 + 7(2j-2) - 12 - 7k - 7(j-1) - 6 + 7(2j) + 7k - 7j + 1 + 5 + 7(2j) - 12 + 3 - 7(2j) + 3 + 1 = 0$.
	
	If $i = 2j+1 \leq 2k+1$, then $S_i = a_{i-3} + b_{i-2} + c_{i-1} + d_i + e_i + f_i + g_i = -7(2j-2) + 3 + 7(2j-1) - 12 + 7k - 7j + 1 + 7(2j+1) - 7k - 7j - 6 + 5 + 7(2j+1) - 12 + 3 - 7(2j+1) + 3 + 1 = 0$.
	
	If $i = 2k+2$, then $S_i = a_{i-3} + b_{i-2} + c_{i-1} + d_i + e_i + f_i + g_i = -7(2k-1) + 3 + 7(2k) - 12 - 7k - 7k - 6 - 28k + 7(2k+2) - 14 + 7k - 7(k+1) + 1 + 5 + 7(2k+2) - 12 + 3 + 28k - 7(2k+2) + 17 + 1 = 0$.
	
	If $i = 2k+3$, then $S_i = a_{i-3} + b_{i-2} + c_{i-1} + d_i + e_i + f_i + g_i = -7(2k) + 3 + 7(2k+1) - 12 + 7k - 7(k+1) + 1 - 28k + 7(2k+3) - 14 + 21k - 7(k+1) + 8 + 5 - 28k + 7(2k+3) - 26 + 3 + 28k - 7(2k+3) + 17 + 1 = 0$.
	
	If $i = 2k+4$, then $S_i = a_{i-3} + b_{i-2} + c_{i-1} + d_i + e_i + f_i + g_i = -7(2k+1) + 3 + 7(2k+2) - 12 + 21k - 7(k+1) + 8 - 28k + 7(2k+4) - 14 + 7k - 7(k+2) + 1 + 5 - 28k + 7(2k+4) - 26 + 3 + 28k - 7(2k+4) + 17 + 1 = 0$.
	
	If $i = 2j+1$ with $2k+4 < i < 4k+2$, then $S_i = a_{i-3} + b_{i-2} + c_{i-1} + d_i + e_i + f_i + g_i = 28k - 7(2j-2) + 17 - 28k + 7(2j-1) - 26 + 7k - 7j + 1 - 28k + 7(2j+1) - 14 + 21k - 7j + 8 + 5 - 28k + 7(2j+1) - 26 + 3 + 28k - 7(2j+1) + 17 + 1 = 0$.
	
	If $i = 2j$ with $2k+4 < i < 4k+2$, then $S_i = a_{i-3} + b_{i-2} + c_{i-1} + d_i + e_i + f_i + g_i = 28k - 7(2j-3) + 17 - 28k + 7(2j-2) - 26 + 21k - 7(j-1) + 8 - 28k + 7(2j) - 14 + 7k - 7j + 1 + 5 - 28k + 7(2j) - 26 + 3 + 28k - 7(2j) + 17 + 1 = 0$.
	
	If $i = 4k+2$, then $S_i = a_{i-3} + b_{i-2} + c_{i-1} + d_i + e_i + f_i + g_i = 28k - 7(4k-1) + 17 - 28k + 7(4k) - 26 + 21k - 7(2k) + 8 - 14k - 7 + 7k + 1 + 5 - 28k + 7(4k+2) - 26 + 3 + 28k - 7(4k+2) + 17 + 1 = 0$.
	
	So $S_i = 0$ for all $i \in [1, n]$.
	
	Now consider $S'_i = a_i + b_i + c_i + d_i + e_{i-1} + f_{i-2} + g_{i-3}$. From the definitions of $e_i$, $f_i$, and $g_i$, we see that $S'_i = g_i - 1 + f_i - 3 + e_i - 5 + d_i + c_{i-1} + 5 + b_{i-2} + 3 + a_{i-3} + 1 = g_i + f_i + e_i + d_i + c_{i-1} + b_{i-2} + a_{i-3} = S_i = 0$.
	
	We will now define an $n \times n$ square array $A=[a_{i,j}]$ in which we fill seven consecutive diagonals. Again, the indices of $a$ will be considered as elements of $\mathbb{Z}_n$. For $i \in [1, n]$, we let
	$a_{i, i+3} = a_i,$
	$a_{i, i+2} = b_i,$
	$a_{i, i+1} = c_i,$
	$a_{i, i} = d_i,$
	$a_{i+1, i} = e_i,$
	$a_{i+2, i} = f_i,$ and
	$a_{i+3, i} = g_i$
	with the other cells empty.
Clearly this fills precisely seven adjacent diagonals of $A$ with the elements in $X$ (see Figure \ref{7in10diag}).
	
	Now in row $i$ of $A$, the seven cells filled are $a_{i, i+3} = a_i$, $a_{i, i+2} = b_i$, $a_{i, i+1} = c_i$, $A_{i, i} = d_i$, $a_{i, i-1} = e_{i-1}$, $a_{i, i-2} = f_{i-2}$, and $a_{i, i-3} = g_{i-3}$. The sum of these seven cells is $S'_i = 0$.
	In column $i$ of $A$, the seven cells filled are $a_{i+3, i} = g_i$, $a_{i+2, i} = f_i$, $a_{i+1, i} = e_i$, $a_{i, i} = d_i$, $a_{i-1, i} = c_{i-1}$, $a_{i-2, i} = b_{i-2}$, and $a_{i-3, i} = a_{i-3}$. The sum of these seven cells is $S_i = 0$.
Therefore array $A$ is a diagonal $SMS(n;7)$.
\end{proof}
	
	\begin{figure}[ht]
		$$\begin{array}{|c|c|c|c|c|c|c|c|c|c|}
		\hline
		7 & -20 & -5 & -4 & & & & 18 & -16 & 20 \\ \hline
		-15 & 14 & 8 & 2 & -11 & & & & 11 & -9 \\ \hline
		-2 & 13 & 21 & -27 & 9 & -18 & & & & 4 \\ \hline
		-3 & 5 & -22 & 28 & 1 & 16 & -25 & & & \\ \hline
		& -10 & 12 & 6 & 35 & -34 & 23 & -32 & & \\ \hline
		& & -17 & 19 & -29 & -28 & -6 & 30 & 31 & \\ \hline
		& & & -24 & 26 & -1 & -21 & 29 & -33 & 24 \\ \hline
		17 & & & & -31 & 33 & 34 & -14 & -13 & -26 \\ \hline
		-19 & 10 & & & & 32 & -30 & -8 & -7 & 22 \\ \hline
		15 & -12 & 3 & & & & 25 & -23 & 27 & -35 \\ \hline
		\end{array}$$
		
		\caption{A diagonal $SMS(10;7)$, using the method of Lemma \ref{7in2mod4diag}}
		\label{7in10diag}
	\end{figure}

Now, we can solve the last quarter of the square case.

\begin{thm}
	\label{oddineven}
	Given $n > t > 2$ with $n$ even and $t$ odd, there exists an $SMS(n;t)$. If $t > 3$ or $n$ is a multiple of $4$, this square is also diagonal.
\end{thm}

\begin{proof}
	If $t = 3$ and $n \equiv 2 \pmod 4$, then we apply Lemma \ref{3ineven}.
Otherwise, we will proceed by induction. As our base case, let $t = 3$ with $n$ a multiple of $4$, or $t = 5$ with $n \equiv 0$ or $2 \pmod 4$, or $t = 7$ with $n \equiv 2 \pmod 4$. In these cases, we apply Lemma \ref{3in0mod4diag}, \ref{5in0mod4diag}, \ref{5in2mod4diag}, or \ref{7in2mod4diag}, respectively.
	
	As our inductive case, assume that there exists a diagonal $SMS(n;t-4)$. Lemma \ref{add4tot} then gives us a diagonal $SMS(n;t)$. We conclude the proof via induction on $t$.
\end{proof}

Lastly, we tie all five of our theorems on square arrays together into the following statement.

\begin{thm}
	\label{square}
	There exists an $SMS(n;t)$ for $n \geq t \geq 1$ precisely when $n, t = 1$ or $n, t > 2$.
\end{thm}

\begin{proof}
	To determine whether an $SMS(n;t)$  exists for a given $t$, $n$ with $n \geq t$, one may consult the above table for an answer as well as which theorem to apply to find it.
	\begin{figure}
	\begin{center}
		{\tabulinesep=2mm
			\begin{tabu}{ c | c | c | c | c |}
				& $n = 1$ & $n = 2$ & $n > 2$ odd & $n > 2$ even \\ \hline
				$t = 1$ & Yes, trivially & No, Theorem \ref{12inn} & No, Theorem \ref{12inn} & No, Theorem \ref{12inn} \\ \hline
				$t = 2$ & \cellcolor{gray!25} $(n < t)$ & No, Theorem \ref{12inn} & No, Theorem \ref{12inn} & No, Theorem \ref{12inn} \\ \hline
				$t > 2$ odd & \cellcolor{gray!25} $(n < t)$ & \cellcolor{gray!25} $(n < t)$ & Yes, Theorem \ref{oddinodd} & Yes, Theorem \ref{oddineven} \\ \hline
				$t > 2$ even & \cellcolor{gray!25} $(n < t)$ & \cellcolor{gray!25} $(n < t)$ & Yes, Theorem \ref{eveninodd} & Yes, Theorem \ref{evenineven} \\ \hline
			\end{tabu}
		}
	\end{center}
	\caption{The various cases of signed magic squares and their corresponding lemmata.}
	\label{signed magic square table}
	\end{figure}
\end{proof}

\section{Signed magic rectangles}\label{SMR}

A natural question to ask is whether the results proven above for signed magic squares extend to signed magic rectangles, i.e arrays where the number of elements in each row differs from the number of elements in each column. A particular case that seems natural to consider is an $n \times 2n$ array that contains $t$ entries in every column and $2t$ entries in every row.

\begin{thm} \label{Heffter rectangles} Let $m \geq t \geq 3$ and suppose $mt \equiv 0 \text{ or } 3 \pmod 4$. Then there exists an $SMA(m,2m;2t,t)$.
\end{thm}

\begin{proof}
Note that by Theorem \ref{Heffterwithemptycells} there exists an integer $m \times m$ Heffter array $A=[a_{i,j}]$ with $t$ entries filled in each row and column. Let $A'=[a_{i,j}']$ be the integer Heffter array, where $a'_{i,j}=-a_{i,j}$ if the cell $(i,j)$ in $A$ is filled, otherwise the cell $(i,j)$ is left blank. Now let $B=[b_{i,j}]$ be the $m \times 2m$ array defined by $b_{i,j}= a_{i,j}$ if $j \leq m$ and $b_{i,j}=a'_{i,j-m}$ if $m<j\leq 2m$. If the cell $(i,j)$ in $A$ is empty the cells $(i,j),(i,m+j)$ in $B$ are also empty (see Figure \ref{Heffter rectangles}).
It is easy to see that $B$ is an $SMA(m,2m;2t,t)$.
\end{proof}

\begin{figure}[ht]
$$\begin{array}{|c|c|c|c|c|c|c|c|}
\hline
4 & 8 & & -12 &-4 & -8 & & 12\\
\hline
-9 & 3 & 6  & & 9 & -3 & -6 &\\
\hline
 & -11 & 1 & 10 & & 11 & -1 & -10\\
\hline
5 & & -7 & 2 & -5 & & 7 & -2\\
\hline
\end{array}$$

\caption{An $SMA(4,8;6,3)$, using the method of Theorem \ref{Heffter rectangles}.}
\label{6 by 3 in 4 by 8}
\end{figure}

\begin{lem}\label{shiftable-t-in-m}
There exists a shiftable $SMS(m;t)$ if and only if $t$ is even and $m\geq t \geq 4.$
\end{lem}

\begin{proof} Suppose $t$ is odd and that there exists an $SMS(m;t)$. Then each row and column of the array contains $t$ filled cells. Since $t$ is odd, there cannot be an equal number of positive and negative entries in each row and column, and the array is not shiftable. If $t=2,$ then clearly there does not exist an $SMS(m;t).$

Now suppose $t \geq 4$ is even. First we consider the case $t=m$ and proceed with induction on $t$.
For the base case, note that both the $4 \times 4 $ and $6 \times 6$ arrays used in the construction of Lemma \ref{evenxeven} are shiftable. So suppose that there exists a shiftable $SMS(t-4;t-4).$ Then we can add four columns to this array by attaching a series of shiftable $2 \times 4$ arrays to the original array. The resulting array is shiftable because $4$ entries, $2$ negative and $2$ positive, are added to each row, and each integer is paired with its opposite in the added columns. Next, we can add four rows to this array by attaching a series of shiftable $4 \times 2$ arrays to the $t-4 \times t $ array. It is easy to see that the resulting array is shiftable.
Hence, by induction, there exists a shiftable $SMS(t;t).$

Now we consider the case $t<m.$ Let $t \equiv 0 \pmod 4$. We again proceed by induction on $t.$ By Lemma \ref{4inndiag}, there exists a shiftable diagonal $SMS(m;4)$ for all $m\geq 4.$ Now let $4 \leq t \leq m-5$ and suppose there exists a shiftable $SMS(m;t)$. We can fill four additional adjacent diagonals using the original $4$-diagonal array shifted appropriately and permuting the columns as necessary. This gives an $SMS(n;t+4),$ and this array is shiftable because we have added $2$ positive and $2$ negative entries to each row and column. By induction, the claim holds when $t \equiv 0 \pmod 4$.

The proof is essentially identical in structure in the cases when $t \equiv 2 \pmod 4$ and $m$ is even, and when $t \equiv 2 \pmod 4$ and $m$ is odd. Note that the base cases are given by Lemma \ref{6in2mod4diag} and Lemma \ref{6inodddiag}, respectively. This completes the proof.
\end{proof}

\begin{figure}[ht]
$$\begin{array}{|c|c|c|c|c|c|c|c|c|c|c|c|c|c|}
\hline
1 & -1 & & 4 &-4 & 6 & -6 & 22 & -22 & & 25 & -25 & 27 & -27\\
\hline
-7 & 20 & -20  & & 12 & -12 & 7 & -28 & 41 & -41 & & 33 & -33 & 28\\
\hline
 19 & -19 & 9 & -9 & & 17 & -17 & 40 & -40 & 30 & -30 & & 38 & -38\\
\hline
-8 & 11 & -11 & 14 & -14 & & 8 & -29 & 32 & -32 & 35 & -35 & & 29\\
\hline
13 & -13 & 3 & -3 & 5 & -5 & & 34 & -34 & 24 & -24 & 26 & -26 &\\
\hline
& 2 & -2 & 15 & -15 & 10 & -10 & & 23 & -23 & 36 & -36 & 31 &-31\\
\hline
-18 & & 21 & -21 & 16 & -16 & 18 & -39 & & 42 & -42 & 37 &-37 & 39\\
\hline
\end{array}$$

\caption{An $SMA(7,14;12,6)$ obtained using the method of Theorem \ref{Heffter rectangles}.}
		\label{12 by 6 in 7 by 14}
	\end{figure}

\begin{thm} \label{shiftable rectangles} Let $m \geq t \geq 3$ with $t$ even. Then there exists an $SMA(m,2m;2t,t)$.
\end{thm}

\begin{proof} By Lemma \ref{shiftable-t-in-m},  there exists a shiftable $SMS(m;t)$, say $A=[a_{i,j}]$.
Now, let $A'=[a_{i,j}']$ be the array defined by $a'_{i,j}=a_{i,j}+\frac{mt}{2}$ if $a_{i,j}>0$,  $a'_{i,j}=a_{i,j}-\frac{mt}{2}$ if $a_{i,j}<0$, and the cell $(i,j)$ is left blank if and only if the corresponding cell in $A$ is left blank. Note that because $A$ is shiftable, $A'$ also has the zero-sum property in its rows and columns.

Now let $B$ be the $m \times 2m$ array where $b_{i,j}= a_{i,j}$ if $j \leq m$ and $b_{i,j}=a'_{i,j-m}$ if $m<j\leq 2m$. If the cell $(i,j)$ is empty in $A$, the cells $(i,j)$ and $(i,j+m)$ are left empty in $B$
(see Figure \ref{12 by 6 in 7 by 14}
It is easy to see that $B$ is the desired array.
\end{proof}

\begin{figure}
{\small
\begin{center}
		{\tabulinesep=2mm
			\begin{tabu}{ c | c | c | c | c |}
				& $m \equiv 0 \pmod 4$ & $m \equiv 1 \pmod 4$ & $m \equiv 2 \pmod 4$ & $m \equiv 3 \pmod 4$ \\ \hline
				$t \equiv 0 \pmod 4$ & Yes, Theorem \ref{Heffter rectangles} & Yes, Theorem \ref{Heffter rectangles} & Yes, Theorem \ref{Heffter rectangles} & Yes, Theorem \ref{Heffter rectangles} \\ \hline
				$t \equiv 1 \pmod 4$ & Yes, Theorem \ref{Heffter rectangles}  & ? & ? & Yes, Theorem \ref{Heffter rectangles} \\ \hline
				$t \equiv 2 \pmod 4$  & Yes, Theorem \ref{Heffter rectangles} & Yes, Theorem \ref{shiftable rectangles}  & Yes, Theorem \ref{Heffter rectangles} & Yes, Theorem \ref{shiftable rectangles} \\ \hline
				$t \equiv 3 \pmod 4$ & Yes, Theorem \ref{Heffter rectangles} & Yes, Theorem \ref{Heffter rectangles} & ? & ? \\ \hline
			\end{tabu}
		}
	\end{center}
\caption{The existence of some $m \times 2m$ signed magic rectangles. }
\label{existence of signed magic rectangles}
}
\end{figure}

These two theorems actually cover many of the cases for $m \times 2m$ signed magic rectangles.
Figure \ref{existence of signed magic rectangles} summarizes our results on signed magic rectangles of dimensions $m \times 2m$ for $m \geq t \geq 3$.

\end{document}